\documentclass[12pt,a4paper]{amsart}
\usepackage{amssymb,amscd,amsthm,amsmath}
\usepackage{booktabs}
\usepackage{amsfonts}
\usepackage[top=35mm, bottom=35mm, left=30mm, right=30mm]{geometry}
\usepackage[colorlinks=true,citecolor=blue]{hyperref}
\usepackage{mathptmx}
\usepackage{eucal}
\usepackage{graphicx}
\usepackage{mathrsfs}
\usepackage{xcolor}
\usepackage[pagewise]{lineno}\nolinenumbers

\newtheorem{thm}{Theorem}[section]
\newtheorem{cor}[thm]{Corollary}
\newtheorem{lem}[thm]{Lemma}
\newtheorem{prop}[thm]{Proposition}

\theoremstyle{definition}
\newtheorem{df}[thm]{Definition}

\newtheorem{rem}[thm]{Remark}
\numberwithin{equation}{section}
\def\UU{\mathcal{U}}

\def\diam{\text{\rm diam}}

\def\logf{\log\frac{1}{\epsilon}}

\def\over{{ \rm \overline{mdim}}_M(\phi,X,d)}

\def\under{ {\rm\underline{mdim}}_M(\phi,X,d)}

\def \logf{\log \frac{1}{\epsilon}}

\def\un{ \underline{ \rm mdim}_M(\phi_1,X,d)}

\def\e{\mu \in \mathcal{E}_{\phi}(X)}

\def\invariant{\mu \in
 \mathcal{M}_{\phi_1}(X)}

\makeatletter

\newcommand{\Rmnum}[1]{\expandafter\@slowromancap\romannumeral #1@}
\makeatother

\begin{document}

\title{Metric mean dimension of flows}

\author{Rui Yang$^{1,2}$, Ercai Chen$^{3}$ and Xiaoyao Zhou$^{*3}$
}
\address
{1.College of Mathematics and Statistics, Chongqing University, Chongqing 401331, P.R.China}
\address{2. Key Laboratory of Nonlinear Analysis and its Applications (Chongqing University), Ministry of Education}
\address
{3.School of Mathematical Sciences and Institute of Mathematics, Ministry of Education Key Laboratory of NSLSCS, Nanjing Normal University, Nanjing, 210023, Jiangsu, P.R.China}

\email{zkyangrui2015@163.com}
\email{ecchen@njnu.edu.cn}
\email{zhouxiaoyaodeyouxian@126.com}

\renewcommand{\thefootnote}{}

\footnotetext{*Corresponding author}



 
\renewcommand{\thefootnote}{\arabic{footnote}}
\setcounter{footnote}{0}

\begin{abstract}
We introduce metric mean dimension to characterize the complexity of continuous flows with infinite topological entropy and derive some of their elementary properties. The main results of this paper are as follows:

$(1)$ We prove that different types of measure-theoretic $\epsilon$-entropies of ergodic measures have the same divergent rate, and establish analogous variational principles for metric mean dimension in terms of the measure-theoretic $\epsilon$-entropies of continuous flows.

$(2)$ We introduce the measure-theoretic metric mean dimension of invariant measures and link it to metric mean dimension via a classical-type variational principle.

$(3)$ Using the Carathéodory-Pesin structure, we introduce  Bowen metric mean dimension on subsets and establish the variational principle for  Bowen metric mean dimension of any non-empty compact subset. Furthermore, we show that the Bowen metric mean dimension of the set of generic points of an ergodic measure is bounded  by the divergent rate of the upper and lower Brin-Katok $\epsilon$-entropies.
\end{abstract}


\maketitle
\pagestyle{plain}  
\tableofcontents 

\section{Introduction}

Let $(X, T)$ denote a  discrete topological dynamical system  (TDS for short), where $T: X \to X$ is a homeomorphism on a compact metrizable space $X$. A central task of dynamics is to capture the  dynamical behavior of systems and characterize its topological complexity. Borrowed from the concept of uncertainty in information theory, topological entropy and measure-theoretic entropy were introduced as two critical topological invariants for quantifying the complexity of $(X, T)$. The classical variational principle (cf. \cite{w82})  provides a fundamental bridge between topological dynamics and ergodic theory, which states that
\[
h_{{top}}(X, T) = \sup_{\mu \in \mathcal{M}(X, T)} h_{\mu}(T),
\]
where $h_{\text{top}}(X, T)$ denotes the topological entropy of $(X,T)$, $h_{\mu}(T)$ denotes the measure-theoretic entropy of $\mu$, and  $\mathcal{M}(X,T)$  is the of $T$-invariant Borel probability measures on $X$.   Resembling the Hausdorff dimension in fractal geometry,  Bowen \cite{b73}  introduced the Bowen topological entropy $h_{\text{top}}^B(T, Z)$ for any subset $Z$ of $X$ in his seminal work, and  characterized the ``size" of the set of generic points of an ergodic measure $\mu$ (denoted by $G_{\mu}$) by showing  that $$h_{\text{top}}(T, G_{\mu}) = h_{\mu}(T)$$ for every ergodic measure $\mu$ on $X$.
Readers can turn to   the  monographs \cite{p97, b08} for a detailed presentation involving the application of Bowen topological entropy in the multifractal analysis of fractal-like sets and the dimension theory of dynamical systems. In 2012, for  every non-empty compact subset  $K \subset X$, Feng and Huang \cite{fh12} established a classical-type variational principle for the Bowen topological entropy of $K$, i.e.,
\[
h_{{top}}^B(T, K) = \sup\left\{\underline{h}_{\mu}^{BK}(T): \mu \in \mathcal{M}(X), \mu(K) = 1\right\},
\]
where $\underline{h}_{\mu}^{BK}(T)$ denotes the  lower Brin-Katok entropy of the Borel probability measure $\mu$ on $X$. 

It should be noted that classical entropy-like quantities are missing its role in describing the complexity of infinite entropy systems for the reason that we gain no information about the dynamics. In 1999, Gromov \cite{gromov} introduced a new topological invariant, mean dimension, for topological dynamical systems. It turns out that this new refined  invariant provides  a powerful tool in solving the embedding problems of dynamical systems \cite{lw00, g15, glt16, g17, gt20}. Another metric-dependent quantity, called metric mean dimension, which is nalogous to the box dimension in fractal geometry,  was introduced  by Lindenstrauss and Weiss \cite{lw00}. It was proved that the  mean dimension is always bounded above by the metric mean dimension, and hence reveals that both types of mean dimensions are zero for finite entropy systems. The classical variational principle motivates  a quite active theme in entropy theory,  injecting ergodic-theoretic ideas into mean dimension theory by establishing new variational principles for metric mean dimension. In 2018, by introducing  the rate-distortion functions from information theory, the pioneering work by  Lindenstrauss and Tsukamoto \cite{lt18} gives  variational principle for metric mean dimension, whose forms are similar to the classical ones, i.e.,
$${\overline{\rm  mdim}}_M(T,X,d)=\limsup_{\epsilon \to 0}\frac{1}{\logf}\sup_{\mu \in \mathcal{M}(X,T)}R_{\mu, L^{\infty}}(\epsilon).$$
where  ${ \overline{\rm  mdim}}_M(T,X,d)$ denotes the upper metric mean dimension of $(X,d,T)$,  and $R_{\mu, L^{\infty}}(\epsilon)$ is the  the $L^{\infty}$ rate-distortion functions of $\mu$. Such  Lindenstrauss-Tsukamoto variational principle still  remains valid when the  rate-distortion function is replaced by other types of measure-theoretic $\epsilon$-entropies given in dynamical systems \cite{vv17, gs20, shi, ycz25}.

Parallel to the entropy theory of $\mathbb{Z}$-actions, for a continuous dynamical system $(X, \phi)$, which we understand  as  a family  $(\phi_t)$ of homeomorphisms for $X$ to $X$,  a corresponding variational principle for the topological entropy (or pressure) of flows was  given  by Bowen and Ruelle \cite{br75}, and  further discussion on the uniqueness of equilibrium states is stated in \cite{fra97}. The authors in   \cite{ps22, WCLW20} related  the Bowen topological entropy of the set of generic points of flows to the measure-theoretic entropy of time-one map.  The  Brin-Katok entropy formula \cite{sz12} and Katok entropy formula \cite{s01} of flows provide two distinct approaches to characterize the measure-theoretic entropy of the time-one map, and play a vital role in the study of sensitivity and entropy theory of flows. In contrast, little is known about the variational principles for the infinite entropy of flows. To some extent, this is because there are intrinsic differences between the ergodic theory of flows and their discrete counterparts. For instance, an invariant probability measure for the time-one map is generally not invariant for the flow, and an ergodic probability measure for the flow is not necessarily ergodic for the time-one map. \emph{Therefore, the corresponding mean dimension theory of flows is not yet well-developed.}

 This paper  aims to develop the mean dimension theory of flows, which includes introducing metric mean dimension for continuous flows and derive some of its fundamental properties, establishing the corresponding variational principles for metric mean dimension and presenting the dimensional characterization of metric mean dimension. Let $$\mathcal{E}=\{\overline{h}_{\mu}^K(\phi,d,\epsilon, \delta), \underline{h}_{\mu}^K(\phi,d,\epsilon, \delta),\overline{h}_{\mu}^K(\phi,d,\epsilon), \underline{h}_{\mu}^K(\phi,d,\epsilon),\overline{h}_{\mu}^{BK}(\phi,d,\epsilon),\underline{h}_{\mu}^{BK}(\phi,d,\epsilon)\}$$
denotes the candidates of measure-theoretic $\epsilon$-entropies of flows\footnote[1]{The superscripts ``K" and ``BK"  stands for the Katok $\epsilon$-entropy and Brin-Katok $\epsilon$-entropy, respectively. The parameter takes value in $(0,1)$, and  Theorem \ref{thm 1.1} holds for every such $\delta$, despite  we do not emphasize  it by  verifying the range of $\delta$.}, whose limits go to the measure-theoretic entropy of time-one map  as $\epsilon \to 0$ whenever $\mu$ is a $\phi$-ergodic measure. In the context of continuous flows with infinite  measure-theoretic entropy, we extend the l entropy formulae of flows to the (Brin-Katok or Katok) $\epsilon$-entropies of ergodic measures.

\begin{thm}\label{thm 1.1}
Let $(X,d,\phi)$ be a flow and $\mu \in \mathcal{E}_{\phi}(X)$ be a $\phi$-ergodic Borel probability measure on $X$.  Then for every  measure-theoretic $\epsilon$-entropy  $h_{\mu}(\phi,d,\epsilon) \in \mathcal{E}$, 
\begin{align*}
\limsup_{\epsilon \to 0}\frac{1}{\logf}h_{\mu}(\phi,d,\epsilon)&=\limsup_{\epsilon \to 0}\frac{1}{\logf}\inf_{\diam(\alpha,d_1) \leq \epsilon} h_{\mu}(\phi_1,\alpha),\\
\liminf_{\epsilon \to 0}\frac{1}{\logf}h_{\mu}(\phi,d,\epsilon)&=\liminf_{\epsilon \to 0}\frac{1}{\logf}\inf_{\diam(\alpha,d_1) \leq \epsilon}h_{\mu}(\phi_1,\alpha),
\end{align*}
where  $d_1(x,y)=\max_{s\in [0,1]}\limits  d(\phi_s x, \phi_s y)$ is the $1$-th Bowen metric on $X$, and  $ h_{\mu}(\phi_1,\alpha)$ denotes the measure-theoretic entropy of  $\alpha$  with respect to  the discrete sample $(X,\phi_1,\mu)$.
\end{thm} 

Besides, we establish  variational principles for the metric mean dimension in terms of the aforementioned measure-theoretic $\epsilon$-entropies of flows.

\begin{thm}\label{thm 1.2}
Let $(X,d,\phi)$ be a flow.  Then for every  $h_{\mu}(\phi,d,\epsilon) \in \mathcal{E}$, 
\begin{align*}
\over&= \limsup_{\epsilon \to 0}\frac{1}{\logf}\sup_{\mu \in \mathcal{E}_{\phi}(X)}h_{\mu}(\phi,d,\epsilon),\\
&=\limsup_{\epsilon \to 0}\frac{1}{\logf}\sup_{\mu \in \mathcal{M}_{\phi}(X)}h_{\mu}(\phi,d,\epsilon)\\
&=\limsup_{\epsilon \to 0}\frac{1}{\logf}\sup_{\mu \in \mathcal{M}(X)}h_{\mu}(\phi,d,\epsilon),
\end{align*} 
and 
\begin{align*}
\over&=\limsup_{\epsilon \to 0}\frac{1}{\logf}\sup_{\mu \in \mathcal{M}_{\phi}(X)}\inf_{\diam (\alpha,d_1) \leq \epsilon}\limits h_\mu(\phi_1,\alpha),\\
&=\limsup_{\epsilon \to 0}\frac{1}{\logf}\sup_{\mu \in \mathcal{E}_{\phi}(X)}\inf_{\diam (\alpha,d_1) \leq \epsilon}\limits h_\mu(\phi_1,\alpha),
\end{align*}
where  $\overline{\rm mdim}_M(\phi,X,d)$ denotes  the   upper  metric mean dimensions of $(X,d,\phi)$, and  $\mathcal{M}(X)$, $\mathcal{M}_{\phi}(X)$, $\mathcal{E}_{\phi}(X)$ are the sets of  Borel probability measures on $X$, $\phi$-invariant, $\phi$-ergodic Borel probability measures on $X$,  respectively. 

The variational principles  are also valid for lower metric mean dimension  $\under$ by changing $\limsup_{\epsilon \to 0}$ into $\liminf_{\epsilon \to 0}$. 
\end{thm}

We remark that,  in particular, if $h_{\mu}(\phi, d, \epsilon):= \overline{h}_{\mu}^K(\phi, d, \epsilon, \delta)$ is given as in Theorem \ref{thm 1.2},   for the variational principle  proved in \cite[Theorem C]{cl21}, the extra assumption of a certain Condition A and  the required uniformly Lipschitz continuous flows  can be removed, and can be strengthened  by showing that the supremum  suffices to take over the set of $\phi$-ergodic measures. However,  to obtain the classical-type variational principle for metric mean dimension,  inspired by the ideas from \cite{yz25} we introduce the measure-theoretic metric mean dimension of invariant measures, which is independent of the choice of the candidate in $\mathcal{E}$,  and link it to metric mean dimension by  establishing  the following variational principle:

\begin{thm}\label{thm 1.3}
Let $(X,d,\phi)$ be a flow. Then
\begin{align*}
\over&=\max_{\mu \in \mathcal{M}_{\phi}(X)}\overline{\rm mdim}_\mu(\phi,d).
\end{align*}
where  $\overline{\rm mdim}_\mu(\phi,d)$  denotes the  measure-theoretic upper  metric mean dimension of $\mu$.
\end{thm}

Based on the work  \cite{b73,ps22,w21}, we introduce the Bowen metric mean dimension of flows on subsets. At the same time, this allows us to describe the ``size" of the set of generic points of ergodic measures in infinite entropy systems. As an application of Theorem \ref{thm 1.1}, we get an inverse variational principle for  Kolmogorov-Sinai $\epsilon$-entropy, showing it can be determined by the $\epsilon$-Bowen topology entropy of  flows over those Borel set with full measure. 

\begin{thm}\label{thm 1.5}
	Let $(X,d,\phi)$ be a flow and $\mu \in \mathcal{E}_{\phi}(X)$.  Then
	\begin{align*}
		\limsup_{\epsilon \to 0}\frac{1}{\logf}\inf_{\diam (\alpha,d_1) \leq \epsilon}\limits h_\mu(\phi_1,\alpha)&=\limsup_{\epsilon \to 0}\frac{1}{\logf}\min_{\text{Borel set Z of $X$ with}~\mu(Z)=1} h_{top}^B(\phi,Z,d,\epsilon),\\
		\liminf_{\epsilon \to 0}\frac{1}{\logf}\inf_{\diam (\alpha,d_1) \leq \epsilon}\limits h_\mu(\phi_1,\alpha)&=\liminf_{\epsilon \to 0}\frac{1}{\logf}\min_{\text{Borel set  $Z$ of $X$ with}~\mu(Z)=1} h_{top}^B(\phi,Z,d,\epsilon),
	\end{align*}
	where  $h_{top}^B(\phi,Z,d,\epsilon)$ denotes the $\epsilon$-Bowen topological entropy of $Z$, and the minimal can be realized by   the set $G_{\mu} \cap E$ with  $G_\mu$ being the set of generic  points of $\mu$ and  any full $\mu$-measure $E$.
\end{thm}

Besides, a proper extension of  variational principle for  metric mean dimension on non-empty compact subsets (not necessarily $\phi$-invariant) is presented as follows:

\begin{thm} \label{thm 1.4}
Let $(X,d,\phi)$ be a flow, and let $K \subset X$ be a non-empty compact  subset. Then
\begin{align*}
\overline{\rm mdim}_M^B(\phi,K,d)&=\max\{\overline{\rm mdim}_{\mu}^{BK}(\phi,d):\mu \in \mathcal{M}(X),\mu(K)=1\},
\end{align*}
where $\overline{\rm mdim}_M^B(\phi,K,d)$ denotes  the  Bowen upper  metric mean dimension on $K$, and  $\overline{\rm mdim}_{\mu}^{BK}(\phi,d)$ is the  upper mean  Brin-Katok local entropy of $\mu$.
\end{thm}

The rest of this paper is organized as follows. In section 2, we introduce the topological and measure-theoretic metric mean dimension,  derive several fundamental properties for them and  prove Theorems \ref{thm 1.1}, \ref{thm 1.2}, and \ref{thm 1.3}. In section 3, we introduce the Bowen metric mean dimension on subsets and prove Theorems \ref{thm 1.5} and \ref{thm 1.4}. In section 4, we develop a unified approach to formulate Bowen entropy-like quantities on subsets and establish the corresponding variational principles for these quantities. Consequently, Theorem \ref{thm 1.4}, as well as many previous variational principles for Bowen topological entropy and Bowen metric mean dimension, is a immediate consequence of our approach.

\section{Metric mean dimension of continuous flows}


\subsection{Metric mean dimension of  flows}

In this subsection, we  introduce the metric mean dimension  of flows and derive some elementary properties for it.

Let  $X$ be  a  compact metrizable  topological space.  A  continuous \emph{flow} (or  $\mathbb{R}$-action) on $X$  is a pair  $(X,\phi)$, where  $ \phi: X \times \mathbb{R} \rightarrow X$  is a continuous mapping  such that

 $(1)$ $\phi(x,0)=x$ for all $x\in X$; 
 
 $(2)$ $\phi_{t+s}=\phi_t\circ \phi_s$ for all $t,s \in \mathbb{R}$, where $\phi_t(x):=\phi(x,t)$ is  a   homeomorphism on $X$. 
 
 We sometimes use a triplet  $(X,d,\phi)$ to emphasize the metric $d$ on $X$. Recall that a discrete topological dynamical system (TDS for short) is  a triplet $(X,d,T)$ with a  compact metric space $(X,d)$ and  a homeomorphism map $T: X \rightarrow X$. The cocycle property of flows implies that    a flow $(X,d,\phi)$  defines a family of TDSs $\{(X,d,\phi_t)\}_{t\in \mathbb{R}}$.

Let $(X,d)$  be a compact  metric space, and  fix  a non-empty subset $Z\subset X$ and $\epsilon >0$. A set $E\subset X$ is  \emph{a $(d,\epsilon)$-spanning set of $Z$} if  for any $x \in Z$, there  exists  $y\in E$ such that $d(x,y)\leq\epsilon.$ Denote by  $r(Z,d,\epsilon)$ the smallest cardinality of $(d,\epsilon)$-spanning set of $Z$. A set $F\subset Z$ is  \emph{a $(d,\epsilon)$-separated set of $Z$} if $d(x,y)>\epsilon$ for any distinct $x,y \in F$. Denote by  $s(Z,d,\epsilon)$ the  largest  cardinality  of $(d,\epsilon)$-separated sets of $Z$.

For \( t > 0 \) and \( x, y \in X \), we define the \( t \)-th Bowen metric \( d_t \) on \( X \) by
\[
d_t(x, y) := \max_{s \in [0, t]} d(\phi_s x, \phi_s y).
\]
The $\epsilon$-topological entropies of $X$ are  given by 
\begin{align*}
r(\phi, X,d,\epsilon)&=\limsup_{t\to \infty} \frac{1}{t} \log r(X,d_t,\epsilon),\\
s(\phi, X,d,\epsilon)&=\limsup_{t\to \infty} \frac{1}{t} \log s(X,d_t,\epsilon).
\end{align*}
It is easy to see that $r(\phi, X,d,\epsilon)\leq s(\phi, X,d,\epsilon) \leq r(\phi, X,d,\frac{\epsilon}{2})$ for every $\epsilon >0$.

\begin{df}
We define the  upper and lower  metric mean dimensions of  $(X,d,\phi)$ as
\begin{align*}
\over&=\limsup_{\epsilon\to 0}\frac{r(\phi,X,d,\epsilon)}{\log \frac{1}{\epsilon}}
=\limsup_{\epsilon\to 0}\frac{s(\phi,X,d,\epsilon)}{\log \frac{1}{\epsilon}},\\
\under&=\liminf_{\epsilon\to 0}\frac{r(\phi,X,d,\epsilon)}{\log \frac{1}{\epsilon}}=\liminf_{\epsilon\to 0}\frac{s(\phi,X,d,\epsilon)}{\log \frac{1}{\epsilon}},
\end{align*}
respectively.
\end{df}

If $\over=\under$, we denote by ${\rm mdim}_M(\phi,X,d)$  the  metric mean dimension of $X$. Recall that   the topological entropy   of  $(X,\phi)$ \cite{BW,br75}  is defined as 
$$h_{top}(\phi,X)=\lim\limits_{\epsilon \to 0}r(\phi, X,d,\epsilon)=\sup_{\epsilon>0}r(\phi, X,d,\epsilon).$$
Hence, any continuous flow with finite topological entropy  has   zero   metric mean dimension. The verifies that the metric mean dimension is  a metric-dependent quantity  to characterize  the complexity of flows admitting   infinite topological entropy. 

The metric mean dimension  can be  also computed by  considering  the discrete times which are integers that multiple the fixed step $\tau>0$. 

\begin{prop}\label{prop 2.2}
Let $(X,d,\phi)$ be a flow. Then  for every $\tau >0$,
\begin{align*}
\over&=\limsup_{\epsilon\to 0}\frac{1}{\log\frac{1}{\epsilon}}\limsup_{n\to \infty}\frac{\log r(X,d_{n\tau},\epsilon)}{n\tau}\\
\under&=\liminf_{\epsilon\to 0}\frac{1}{\log\frac{1}{\epsilon}}\limsup_{n\to \infty}\frac{ \log r(X,d_{n\tau},\epsilon)}{n\tau}.
\end{align*}
\end{prop}

\begin{proof}
Fix $\tau >0$. For every $t>0$, there is a  positive integer $n $\footnote[1]{In the rest of this paper, to distinguish any real number ``$t$'' we  are always  accustomed to use   ``$n$'' (or ``$N$") to refer to a  positive integer.} such that $n\tau <t \leq (n+1)\tau$. Then $r(X,d_{n\tau},\epsilon)\leq r(X,d_{t},\epsilon)\leq r(X,d_{(n+1)\tau},\epsilon)$. This yields the desired equalities by taking the corresponding limits.
\end{proof}

Replacing $d_t$ by \emph{$n$-th Bowen metric\footnote[2]{The minuscule difference of the  two kind of Bowen metrics on $X$ lies in the metric $d_{n,\phi_t}$  emphasizes the $n$-th Bowen metric with respect to the time-$t$ map $\phi_t$, while $d_n(x,y)=max_{s\in [0,n]}d(\phi_sx,\phi_sy)$ (without the subscript $\phi_t$) is the $n$-th Bowen metric with respect to $\phi$.} of time $t$-map $\phi_t$} $$d_{n,\phi_t}(x,y):=\max_{ 0\leq j\leq n-1} d(\phi_{tj}x,\phi_{tj} y),$$ where  $\phi_{tj}=\phi_t^{j}=\phi_t\circ\cdots \circ \phi_t$ is  the composition  of the time $t$-map $\phi_t$ for $j$-times, one can define the  upper  and lower metric mean dimensions of $X$ for time-$\tau$ map. For this case, we similarly define  the quantities $r(\phi_t, X,d,\epsilon),$ $ s(\phi_t, X,d,\epsilon)$, and   $\overline{\rm{mdim}}_M(\phi_t, X,d)$,  $\underline{\rm{mdim}}_M(\phi_t, X,d)$, respectively.


We present an analogue of  Abramov-type formula  \cite{a59} for metric mean dimension, which clarifies  the metric mean dimensions of a flow and its discrete samples.

 \begin{prop}\label{prop 2.3}
Let $(X,d,\phi)$ be a flow.  Then for every $\tau>0$,
\begin{align*}
\frac{1}{\tau}\overline{\rm mdim}_M(\phi_{\tau},X, d_\tau) &= \over,\\
\frac{1}{\tau}\underline{\rm mdim}_M(\phi_{\tau},X, d_\tau) &= \under,
\end{align*}
where $d_\tau(x,y):=\max_{s\in [0,\tau]} d(\phi_s x, \phi_s y).$
 \end{prop} 
 
 \begin{proof}
Fix $\tau >0$. For every $n\in \mathbb{{N}}$, observe that for $x,y \in X$,
 \begin{align}\label{inequ 2.1}
 d_{n\tau}(x,y)=\max_{0\leq j \leq n-1}d_{\tau}(\phi_{\tau j}x,\phi_{\tau j}y)=(d_{\tau})_{n,\phi_\tau}(x,y).
 \end{align}
For every $\epsilon >0$, we have $r(X,d_{n\tau},\epsilon)=r(X,(d_{\tau})_{n,\phi_\tau},\epsilon)$. By Proposition \ref{prop 2.2},  we complete the proof.
 \end{proof}
 
 \begin{rem}
 It  remains unclear whether $d_{\tau}$ can be replaced by $d$  for all $\tau >0$.  Since   $d_{\tau} \geq d$, it always holds that  $\frac{1}{\tau}\overline{\rm mdim}_M(\phi_{\tau},X, d) \leq \over$.    By the continuity of $\phi$, there exists $0<\delta(\epsilon)<\epsilon$ such that if $d(x,y)<\delta(\epsilon)$, then $d(\phi_s x,\phi_s y)<\epsilon$ for all $s\in [0,\tau]$. Then  one can deduce that
 $$\frac{r(\phi_{\tau}, X, d,\delta(\epsilon))}{\tau}\geq  r(\phi,X,d,\epsilon).$$
By the definition of metric mean dimension, one  can formulate  the following inequality:
 $$\frac{\log \frac{1}{\delta(\epsilon)}}{\log \frac{1}{\epsilon}}\cdot \frac{r(\phi_{\tau}, X,d,\delta(\epsilon))}{\tau   \cdot \log \frac{1}{\delta(\epsilon)}}\geq \frac{r(\phi,X,d,\epsilon)}{\logf}.$$
It is difficult for authors to determine  whether $\limsup_{\epsilon\to 0}\frac{\log \frac{1}{\delta(\epsilon)}}{\log \frac{1}{\epsilon}} =1$ or not. For certain  uniformly Lipschitz flows,  this question is affirmative.
 \end{rem}

 The following definition is a slightly modification of \cite[Defintiion 4.1]{cl21} whose argument is  to study the scaled pressure of  fixed-point free  flows. 
 
 \begin{df}\label{def 2.3}
A  flow $(X,d, \phi)$ is called a  uniformly Lipschitz flow if for any $t_0>0$,  there exists $L(t_0)>0$ such that for any   $x,y\in X$, one has $$d(\phi_t(x),\phi_t(y))\leq  L(t_0) \cdot d(x,y)$$
for all $t\in [0,t_0]$. 
 \end{df}
  
\begin{prop}
If $(X,d,\phi)$ is a  uniformly Lipschitz flow, then for every $\tau>0$,
\begin{align*}
\frac{1}{\tau}\overline{\rm mdim}_M(\phi_{\tau},X, d) &= \over,\\
\frac{1}{\tau}\underline{\rm mdim}_M(\phi_{\tau},X, d) &= \under.
\end{align*}
\end{prop} 

\begin{proof}
Since $(X,d,\phi)$ is a  uniformly Lipschitz flow,  there exists  a constant $L(\tau)>0$ such that if  $x,y \in X$ with $d(x,y)<\frac{\epsilon}{L(\tau)}$, then $d(\phi_{s}x, \phi_{s}y)<\epsilon$ for all $0\leq s \leq \tau$.  Using the fact 
$r(\phi,X,d,\epsilon)=\limsup_{n\to \infty}\limits\frac{\log r(X,d_{n\tau},\epsilon)}{n\tau},$ 
we have 
$$\frac{r(\phi_{\tau},X,d,\frac{\epsilon}{L(\tau)})}{\tau}\geq  r(\phi,X,d,\epsilon).$$
This implies the desired equalities.
\end{proof}

We proceed to derive a  Lindenstrauss-Weiss's  inequality for continuous flows.  
Gutman and  Jin \cite{gj20} defined  the \emph{mean dimension} of  a  flow $(X,\phi)$ as
$$\text{\rm mdim}(\phi,X)=\lim_{\epsilon\to 0}\lim_{n\to \infty}\frac{\text{Wdim}_{\epsilon}(X,d_n)}{n},$$
where  ${\rm Wdim}_{\epsilon}(X,d_n)$ is  defined by $(d_n,\epsilon)$-embedding mappings of $X$ (cf. \cite[\S 2]{gj20}). 

 We remark that mean dimension  is independent of the choice of the compatible metrics on $X$.  Using the $n$-th Bowen metric $d_{n,\phi_1}$ of  the time-one map,   we similarly define $\text{Wdim}_{\epsilon}(X,d_{n,\phi_1})$  and   the mean dimension $\text{mdim}(\phi_1, X)$   for the time one-map  $\phi_1$. 

\begin{thm}\label{thm 2.7}
Let $(X,\phi)$ be a flow. Then  for any compatible  metric $d$ on $X$,
\begin{align*}
\text{\rm mdim}(\phi,X)\leq \under \leq \over.
\end{align*}
\end{thm}
\begin{proof}
By \cite[Proposition 2.5]{gj20},  we have  $$\text{\rm mdim}(\phi,X)=\text{\rm mdim}(\phi_1,X).$$
Using  the classical inequality $\text{\rm mdim}(\phi_1,X)\leq \text{\underline{\rm mdim}}_M(\phi_1,X, d)$ \cite[Theorem 4.2]{lw00} and  the fact $\un \leq \under$,  we  get the desired inequality. 
\end{proof} 

Finally,   we determine the mean  dimensions of a continuous flow and its induced system. Denote by  $\mathcal{M}(X)$ the  set of Borel probability measures  on $X$, and endow it  with the weak$^{*}$-topology. Given a flow $(X,\phi)$,  it induces a flow $(\Phi_t)_{t\in \mathbb{R}}$ on $\mathcal{M}(X)$ by the homeomorphism map $\Phi_t(\mu)=\mu\circ \phi_{-t}$.

\begin{thm}
Let $(X,d,\phi)$ be a flow.  Then 
\begin{itemize}
\item [(1)] $h_{top}(\phi,X)>0 \Leftrightarrow \text{\rm mdim}(\Phi, \mathcal{M}(X))>0\Leftrightarrow \text{\rm mdim}(\Phi, \mathcal{M}(X))=\infty$;
\item [(2)] for any compatible metric $D$ on $\mathcal{M}(X)$, $h_{top}(\phi,X)>0 \Leftrightarrow \text{\rm mdim}_M(\Phi,\mathcal{M}(X),D)>0\Leftrightarrow \text{\rm mdim}_M(\Phi, \mathcal{M}(X),D)=\infty$.
\end{itemize}
Consequently,  for any compatible metric $D$ on $\mathcal{M}(X)$, $$h_{top}(\phi,X)>0 \Leftrightarrow \text{\rm mdim}(\Phi, \mathcal{M}(X))=\infty\Leftrightarrow \text{\rm mdim}_M(\Phi,\mathcal{M}(X),D)=\infty.$$
\end{thm}

\begin{proof}
(1). It has been proved in  \cite{bs25} that $$h_{top}(\phi_1,X)>0 \Leftrightarrow \text{\rm mdim}(\Phi_1, \mathcal{M}(X))>0\Leftrightarrow \text{\rm mdim}(\Phi_1, \mathcal{M}(X))=\infty.$$
Since $h_{top}(\phi,X)=h_{top}(\phi_1,X)$ and  $\text{mdim}(\Phi, \mathcal{M}(X))=\text{mdim}(\Phi_1, \mathcal{M}(X))$, this shows (1).

(2).  By Theorem \ref{thm 2.7} and (1), it  holds that  $$h_{top}(\phi,X)>0 \Rightarrow \text{\rm mdim}_M(\Phi,\mathcal{M}(X),D)=\infty \Rightarrow \text{\rm mdim}_M(\Phi, \mathcal{M}(X),D)>0.$$
It is clear  that $\text{\rm mdim}_M(\Phi, \mathcal{M}(X),D)>0\Rightarrow h_{top}(\Phi,\mathcal{M}(X))=h_{top}(\Phi_1,\mathcal{M}(X))=\infty$. By  Glasner-Weiss theorem \cite{gw95}, we have $h_{top}(\Phi_1,\mathcal{M}(X))=\infty\Rightarrow  h_{top}(\phi_1,X)>0$. This shows (2).
\end{proof}

\subsection{Measure-theoretic $\epsilon$-entropy of  flows}
\label{subsec 2.2}
In this subsection, we review several measure-theoretic $\epsilon$-entropies of  flows, and  prove Theorem \ref{thm 1.1}.

To define measure-theoretic metric mean dimension of invariant measures, the corresponding role of  topological $\epsilon$-entropy  shall   be replaced by  measure-theoretic $\epsilon$-entropy of invariant measures,  whose limits equal the measure-theoretic entropy of the time-one map as $\epsilon \to 0$. In the following, we recall the precise definitions of  measure-theoretic  $\epsilon$-entropies of  flows.


\subsubsection{Kolmogorov-Sinai $\epsilon$-entropy.}

 Given  $t \in \mathbb{R}$, let $\mathcal{M}_{\phi_t}(X)$, $\mathcal{E}_{\phi_t}(X)$ denote  the sets of \emph{$\phi_t$-invariant}, \emph{$\phi_t$-ergodic} Borel probability measures on $X$, respectively.  One says that  $\mu \in \mathcal{M}(X)$ is  \emph{$\phi$-invariant} if $\mu$ is $\phi_t$-invariant for all $t\in \mathbb{R}$. A $\phi$-invariant measure  is said to be \emph{$\phi$-ergodic} if  any  $\phi$-invariant Borel  set $B$ (e.g. $\phi_t(B)=B$ for all $t\in \mathbb{R}$)   has  measure $0$ or $1$.  Denote by  $\mathcal{M}_{\phi}(X), \mathcal{E}_{\phi}(X)$   the sets of  all  $\phi$-invariant, $\phi$-ergodic probability measures\footnote[3]{Due to the well-known Krein-Milman theorem,  $\mathcal{E}_{\phi}(X)$ is non-empty (cf. \cite[Theorem 3.3.27]{fh20}).}, respectively. Then it is clear that for all $t\in \mathbb{R}$,
 $$\mathcal{E}_{\phi}(X)\subset \mathcal{M}_{\phi}(X)\subset \mathcal{M}_{\phi_t}(X).$$
 Conversely, a $\phi_1$-invariant measure for time-one map  is not necessarily $\phi$-invariant for flow. Besides, a $\phi$-ergodic measure for flow is not necessarily $\phi_1$-ergodic  for time-one map.  However,  each $\mu \in \mathcal{E}_{\phi}(X) $ is $\phi_t$-ergodic  except by at most countably many $t\in \mathbb{R}$ (cf. \cite[Theorem 3.3.13]{fh20}).

 
Let  $\alpha$  be  a finite Borel  partition  of $(X,d)$. The \emph{diameter} of $\alpha$ is defined by ${\diam}(\alpha,d)=\max_{A\in \alpha}{{\diam}(A,d)}$.  Let  $\mu \in \mathcal{M}_{\phi}(X)$.
The  measure-theoretic entropy  of $\mu$   with respect to (w.r.t.) time-one map $\phi_1$ is defined by 
$$h_\mu(\phi_1):=\sup_{\mathcal{\alpha}}h_\mu(\phi_1, \mathcal{\alpha}),$$
where the supremum ranges over all  finite  Borel partitions of $X$, and  $h_\mu(\phi_1, \mathcal{\alpha})$ is the usual measure-theoretic entropy of $\mu$ w.r.t. $\alpha$. 

For every $\mu \in \mathcal{M}_{\phi_1}(X)$ (cf. \cite[Theorem 8.3]{w82}), we have $$h_\mu(\phi_1)=\lim\limits_{\epsilon \to 0} \inf_{\diam (\alpha,d) \leq \epsilon} \limits h_\mu(\phi_1, \alpha).$$
We  define   the Kolmogorov-Sinai $\epsilon$-entropy of $\mu$ w.r.t. $\phi_1$ as
$$\inf_{\diam (\alpha,d) \leq \epsilon}h_\mu(\phi_1, \alpha),$$
where the infimum ranges over all finite Borel partitions  of $X$ with diameter at most $\epsilon$.

\subsubsection{Katok's  $\epsilon$-entropy.}
Let $(X,d)$ be a compact metric space. For $x\in X$ and $\epsilon >0$,  we  define the $d$-open ball of $x$ with radius $\epsilon$ as
$B(x,d,\epsilon):=\{y\in X: d(x,y)<\epsilon\}.$
Let $\mu \in  \mathcal{M}(X)$, $\epsilon>0$, $t>0$ and  $\delta \in (0,1)$. Set
$$R_\mu^\delta(\phi, t,d, \epsilon)=\min\{\#E: E\subset X ~\text{and} ~\mu (\cup_{x\in E}B(x,d_t,\epsilon))\geq 1- \delta\}.$$

We define \emph{the upper and lower  Katok's  $\epsilon$-entropies of $\mu$ w.r.t. $\phi$} as 
\begin{align*}
\overline{h}_{\mu}^K(\phi,d,\epsilon, \delta)&=\limsup_{t\to \infty} \frac{1}{t} \log R_\mu^\delta(\phi, t,d, \epsilon),\\
\underline{h}_{\mu}^K(\phi,d,\epsilon, \delta)&=\liminf_{t\to \infty} \frac{1}{t} \log R_\mu^\delta(\phi, t, d,\epsilon). 
\end{align*}
If $ \e$,  Sun \cite[Theorem A]{s01}  proved that for any  $\delta \in (0,1)$,
$$\lim_{\epsilon\to 0} \overline{h}_{\mu}^K(\phi,d,\epsilon, \delta)=\lim_{\epsilon\to 0}\underline{h}_{\mu}^K(\phi,d,\epsilon, \delta)=h_\mu(\phi_1).$$
For time-$\tau$  map $\phi_\tau$, one can  similarly define   $R_\mu^\delta(\phi_\tau, n, d,\epsilon)$  by using the Bowen metric $d_{n,\phi_\tau}$, and then formulate the corresponding  Katok's $\epsilon$-entropies
$\overline{h}_{\mu}^K(\phi_\tau,d,\epsilon, \delta), 
\underline{h}_{\mu}^K(\phi_\tau,d,\epsilon, \delta)$. 

Notice that the Katok's $\epsilon$-entropies are non-decreasing  when $\delta$ decreases. For every $\mu \in \mathcal{M}(X)$, we can alternatively define   $\overline{h}_{\mu}^K(\phi,d,\epsilon)$, $\underline{h}_{\mu}^K(\phi,d,\epsilon)$, $\overline{h}_{\mu}^K(\phi_\tau,d,\epsilon)$ and $\underline{h}_{\mu}^K(\phi_\tau,d,\epsilon)$ by letting $\delta \to 0$ for the above Katok's $\epsilon$-entropies.

\subsubsection{Brin-Katok's  $\epsilon$-entropy.}
Let $\mu \in \mathcal{M}(X)$ and $\epsilon >0$. Put 
\begin{align*}
\overline{h}_\mu^{BK}(\phi,x,d,\epsilon):&= \limsup_{t\to \infty}-\frac{\log \mu (B(x,d_t,\epsilon))}{t},\\
\underline{h}_\mu^{BK}(\phi,x,d,\epsilon):&= \liminf_{t\to \infty}-\frac{\log \mu (B(x,d_t,\epsilon))}{t}.
\end{align*}
Since the function $x\in X \mapsto \mu (B(x,d_t,\epsilon))$ is lower semi-continuous for every fixed $t>0$,   $\overline{h}_\mu^{BK}(\phi,x,d,\epsilon)$ and $\underline{h}_\mu^{BK}(\phi,x,d,\epsilon)$ are both Borel measurable.

We   define the \emph{upper and lower Brin-Katok  $\epsilon$-entropies of $\mu$  w.r.t. $\phi$} as 
$$\overline{h}_{\mu}^{BK}(\phi, d,\epsilon)=\int \overline{h}_\mu^{BK}(\phi,x,d,\epsilon)d\mu,~
\underline{h}_{\mu}^{BK}(\phi, d,\epsilon)=\int \underline{h}_\mu^{BK}(\phi,x,d,\epsilon)d\mu.$$
If $ \e$,  Shen and Zhao \cite[Proposition 1]{sz12}  proved  that 
$$\lim_{\epsilon\to 0} \overline{h}_{\mu}^{BK}(\phi, d,\epsilon)=\lim_{\epsilon\to 0}\underline{h}_{\mu}^{BK}(\phi, d,\epsilon)=h_\mu(\phi_1).$$

Again, for time-$\tau$  map $\phi_\tau$,  one can  similarly use the Bowen ball $B(x,d_{n,\phi_{\tau}},\epsilon)$  to  define   $\overline{h}_\mu^{BK}(\phi_{\tau},x,d,\epsilon)$ and $\overline{h}_\mu^{BK}(\phi_{\tau},x,d,\epsilon)$, and then formulate the upper and lower Brin-Katok $\epsilon$-entropies $\overline{h}_{\mu}^{BK}(\phi_{\tau}, d,\epsilon)$ and  $\overline{h}_{\mu}^{BK}(\phi_{\tau}, d,\epsilon)$.

\begin{prop}\label{pro 2.9}
Fix $\epsilon>0$ and $\e$. Then for $\mu$-a.e. $x\in X$, we have
$\overline{h}_\mu^{BK}(\phi,x,d,\epsilon)=\overline{h}_\mu^{BK}(\phi,d,\epsilon)$ and $\underline{h}_\mu^{BK}(\phi,x,d,\epsilon)=\underline{h}_\mu^{BK}(\phi,d,\epsilon).$
\end{prop}

\begin{proof} 
It suffices to show that  for $\mu$-a.e. $x\in X$, $$\overline{h}_\mu^{BK}(\phi,x,d,\epsilon)=\overline{h}_\mu^{BK}(\phi,d,\epsilon),$$  since  the remaining statement is similarly obtained. 

Since $\e$, we may  choose $\tau >0$ such that $\mu$ is $\phi_\tau$-ergodic for the TDS $(X,d,\phi_\tau)$. Notice that the continuity of $\phi$ implies that  $d_\tau$ is compatible with the metric $d$. Hence, $\mu$ is ergodic for the TDS $(X,d_\tau,\phi_\tau)$.
Since $$\phi_\tau (B(x,(d_\tau)_{n+1,\phi_\tau},\epsilon))\subset B(\phi_{\tau}x,(d_\tau)_{n,\phi_\tau},\epsilon),$$  by the $\phi_\tau$-invariance of $\mu$ we  have 
$$\mu(B(x,(d_\tau)_{n+1,\phi_\tau},\epsilon))=\mu(\phi_{-\tau}\phi_{\tau}(B(x,(d_\tau)_{n+1,\phi_\tau},\epsilon)))\leq \mu(B(\phi_{\tau}x,(d_\tau)_{n,\phi_\tau},\epsilon)).$$
This implies that $\overline{h}_\mu^{BK}(\phi_{\tau},\phi_{\tau} x,d_{\tau},\epsilon)\leq \overline{h}_\mu^{BK}(\phi_{\tau},x,d_\tau,\epsilon)$.  The $\phi_\tau$-ergodicity of $\mu$ yields that  for $\mu$-a.e. $x\in X$, we have  $\overline{h}_{\mu}^{BK}(\phi_{\tau},x, d_\tau,\epsilon)=\overline{h}_{\mu}^{BK}(\phi_{\tau}, d_\tau,\epsilon)$.
By (\ref{inequ 2.1}),  for every $n\in \mathbb{N}$ we have  
\begin{align*}
B(x,d_{n\tau},\epsilon)=B(x,(d_\tau)_{n,\phi_\tau},\epsilon),
\end{align*}
which implies that  $\overline{h}_{\mu}^{BK}(\phi,x,d,\epsilon)=\frac{1}{\tau}\overline{h}_{\mu}^{BK}(\phi_{\tau},x, d_\tau,\epsilon)$ for every $x\in X$, and hence $\overline{h}_{\mu}^{BK}(\phi ,d,\epsilon)=\frac{1}{\tau}\overline{h}_{\mu}^{BK}(\phi_{\tau}, d_\tau,\epsilon)$. This shows that  
$\overline{h}_\mu^{BK}(\phi,x,d,\epsilon)=\overline{h}_\mu^{BK}(\phi,d,\epsilon)$ for $\mu$-a.e. $x\in X$.
\end{proof}

Analogous to Proposition \ref{prop 2.3}, the following proposition is a measure-theoretic version of Abramov-type formula.

\begin{prop}\label{prop 2.10}
Let $(X,d,\phi)$ be a flow and $\mu \in \mathcal{M}(X)$.  Then for every $\tau>0$,
\begin{align*}
\limsup_{\epsilon \to 0}\frac{1}{\logf}h_{\mu}(\phi,d,\epsilon)&=\frac{1}{\tau}\limsup_{\epsilon \to 0}\frac{1}{\logf}h_{\mu}(\phi_\tau,d_\tau,\epsilon),\\
\liminf_{\epsilon \to 0}\frac{1}{\logf}h_{\mu}(\phi,d,\epsilon)&=\frac{1}{\tau}\liminf_{\epsilon \to 0}\frac{1}{\logf}h_{\mu}(\phi_\tau,d_\tau,\epsilon),
\end{align*}
where $h_{\mu}(\phi,d,\epsilon)\in \{\overline{h}_{\mu}^K,\underline{h}_{\mu}^K, \overline{h}_{\mu}^{BK},\underline{h}_{\mu}^{BK}\}$, and 
$h_{\mu}(\phi_\tau,d_\tau,\epsilon)\in \{\overline{h}_{\mu}^K,\underline{h}_{\mu}^K, \overline{h}_{\mu}^{BK},\underline{h}_{\mu}^{BK}\}$.
\end{prop}

\begin{proof}
By (\ref{inequ 2.1}), for every measure-theoretic $\epsilon$-entropy  we have $$h_{\mu}(\phi,d,\epsilon)= \frac{h_{\mu}(\phi_\tau,d_\tau,\epsilon)}{\tau}$$ for every $\epsilon>0$. This implies the desired equalities.
\end{proof}

\subsubsection{Divergent rates of measure-theoretic $\epsilon$-entropies of flows}

For flows with infinite measure-theoretic entropy, we  verify the divergent rates of the aforementioned measure-theoretic $\epsilon$-entropies of ergodic measures coincide with the Kolmogorov-Sinai $\epsilon$-entropy of the time-one map.

\begin{thm}\label{thm 2.11}
Let $(X,d,\phi)$ be a flow  and $\mu \in \mathcal{E}_{\phi}(X)$.  Then for every  $h_{\mu}(\phi,d,\epsilon) \in \mathcal{E}$, there exist  constants $C_1,C_2>0$ such that for all $\epsilon>0$,
$$\inf_{\diam(\alpha,d_1) \leq C_1\epsilon} h_{\mu}(\phi_1,\alpha) \leq h_{\mu}(\phi,d,\epsilon)\leq \inf_{\diam(\alpha,d_1) \leq C_2\epsilon} h_{\mu}(\phi_1,\alpha),$$
where $\mathcal{E}=\{\overline{h}_{\mu}^K(\phi,d,\epsilon, \delta), \underline{h}_{\mu}^K(\phi,d,\epsilon, \delta),\overline{h}_{\mu}^K(\phi,d,\epsilon), \underline{h}_{\mu}^K(\phi,d,\epsilon),\overline{h}_{\mu}^{BK}(\phi,d,\epsilon),\underline{h}_{\mu}^{BK}(\phi,d,\epsilon)\}.$
\end{thm} 

We invoke two  auxiliary  lemmas to prove Theorem  \ref{thm 2.11}.

Let  $(X,d,T)$ be a TDS. Denote by $\mathcal{M}_{T}(X)$ and $\mathcal{E}_{T}(X)$ the sets of $T$-invariant, $T$-ergodic Borel probability measures on $X$,  respectively.  Given a finite open cover $\mathcal{U}$ of $X$,  by $\diam (\mathcal{U},d):=\max_{U\in \mathcal{U}}\diam(U,d) $ we denote the \emph{diameter} of  $\mathcal{U}$;  by     $\text{Leb}(\mathcal{U},d)$  we denote the \emph{Lebesgue number} of $\mathcal{U}$, that  is  the largest  positive number $\delta >0$ such that  every open ball $B(x,d,\delta)$ of radius $\delta$  is contained in  some element of $\mathcal{U}$.   Let $\mathcal{U}^n:=\vee_{j=0}^{n-1}T^{-j}\mathcal{U}$ denote  the \emph{$n$-th join} of  $\mathcal{U}$. Given $\delta \in (0,1)$ and $\mu \in \mathcal{M}(X)$, we put $$N_{\mu}(\mathcal{U},\delta)=\min\{\#\mathcal{C}: \mathcal{C} \subset \mathcal{U}\text{~and~}\mu( \cup_{C\in \mathcal{C}}C)\geq 1- \delta\},$$
and define  the Shapira's entropy of $\UU$ w.r.t. $\mu$ as
\begin{align*}
\underline h_\mu^S(T,\mathcal{U})&=\lim_{\delta \to 0}\liminf_{n\to \infty}\frac{\log N_{\mu}(\mathcal{U}^n,\delta)}{n}.
\end{align*}

\begin{lem}\cite[Proposition 4.16]{hyz11} \label{lem 2.13}
Let  $(X,d,T)$ be a TDS, and let $\mu   \in  \mathcal{M}_{T}(X) $. Then for every  finite open cover $\UU$ of $X$,
\begin{align}
\inf_{\alpha \succ \UU}h_{\mu}(T,\alpha)\leq \underline h_\mu^S(T,\mathcal{U}),
\end{align}
where $\alpha$ ranges over all finite Borel  partitions of $X$ refining than $\UU$.
\end{lem}

\begin{prop}\label{prop 2.14}
Let  $(X,d,T)$ be a TDS, and let $\mu \in \mathcal{M}_{T}(X)$. Then for every $\epsilon >0$,  we have 
  $$\inf_{\diam(\alpha,d) \leq \epsilon}\limits  h_{\mu}(T,\alpha) \leq  \inf_{\diam(\UU,d)\leq \epsilon} \limits  \underline{h}_\mu^S(T,\mathcal{U}) \leq \underline{h}_{\mu}^K(T,d,\frac{\epsilon}{4}).$$

\end{prop}

\begin{proof}
For every finite open cover $\UU$ of $X$ with diameter  at most $\epsilon$, the diameter of a finite Borel partition  $\alpha$ of $X$  refining than $\UU$ is at most $\epsilon$. By Lemma \ref{lem 2.13}, this implies that  $$\inf_{\diam(\alpha,d) \leq \epsilon} h_{\mu}(T,\alpha) \leq  \inf_{\diam(\UU,d)\leq \epsilon}\underline{h}_\mu^S(T,\mathcal{U}).$$

 Let  $\{B(x,d,\frac{\epsilon}{4}):x\in E\}$ be a  finite  family of open balls, where $E$ is a finite $\frac{\epsilon}{4}$-net of $X$.  Then the family  $\{B(x,d,\frac{\epsilon}{2}):x\in E\}$ is  a finite open cover of $X$  with $\diam{(\mathcal{ U},d)}\leq \epsilon$ and $\text{Leb}(\mathcal{U},d)\geq \frac{\epsilon}{4}$.  Then each open ball $B(x,d,\frac{\epsilon}{4})$ is contained in some element of $\UU$. By comparing the definitions of $ \underline h_\mu^S(T,\mathcal{U})$ and $\underline{h}_{\mu}^K(T,d,\epsilon)$ , we get the  desired inequalities.
\end{proof}

\begin{lem}(\cite[Lemma 3.1]{ycz25}, \cite[Theorem 1]{s26}) \label{lem 2.12}
	Let  $(X,d,T)$ be a TDS, and let $\mu  \in \mathcal{E}_{T}(X)$. Then for every $\epsilon >0$,
	\begin{align}
		&\overline{h}_\mu^{BK}(T,d,2\epsilon)\leq \underline{h}_\mu^{K}(T,d,\frac{\epsilon}{4},\delta) \nonumber  
		\leq \overline{h}_\mu^{K}(T,d,\frac{\epsilon}{4},\delta)
	\\
	\leq& \underline{h}_\mu^{K}(T,d,\frac{\epsilon}{32}\leq 
		\overline{h}_\mu^{K}(T,d,\frac{\epsilon}{32})\leq \underline{h}_\mu^{BK}(T,d,\frac{\epsilon}{4}) \leq  \overline{h}_\mu^{BK}(T,d,\frac{\epsilon}{4}).
	\end{align}
\end{lem}

Now we are ready to prove Theorem \ref{thm 2.11}.

\begin{proof}[Proof of  Theorem  \ref{thm 2.11}]
Since $\mu$ is $\phi$-ergodic,  we choose $\tau >0$ such that $\mu$ is ergodic for the TDS $(X,d_\tau,\phi_\tau)$.  Applying the Lemma \ref{lem 2.12} to  TDS $(X,d_\tau,\phi_\tau)$, and using the fact  that for every $\epsilon >0$, $$h_{\mu}(\phi,d,\epsilon) = \frac{h_{\mu}(\phi_\tau,d_\tau,\epsilon)}{\tau},$$ we have 
\begin{align}\label{inequ 2.6}
&\overline{h}_\mu^{BK}(\phi,d,2\epsilon)\leq \underline{h}_\mu^{K}(\phi,d,\frac{\epsilon}{4},\delta)\nonumber \nonumber 
\leq \overline{h}_\mu^{K}(\phi,d,\frac{\epsilon}{4},\delta)
\\
\leq & \underline{h}_\mu^{K}(\phi,d,\frac{\epsilon}{32})\leq \overline{h}_\mu^{K}(\phi,d,\frac{\epsilon}{32})\leq 
\underline{h}_\mu^{BK}(\phi,d,\frac{\epsilon}{4})\leq \overline{h}_\mu^{BK}(\phi,d,\frac{\epsilon}{4}).
\end{align}
Hence, the divergent rates of Brin-Katok $\epsilon$-entropy and Katok $\epsilon$-entropy  of $\mu$ coincide.

Fix a   finite Borel  partition $\alpha$ of $X$ with $\diam (\alpha,d_1)\leq \epsilon$. Then every atom of   $\alpha_n:=\vee_{j=0}^{n-1}\phi_{-j}\alpha$ to which  $x$  belongs is contained in  $B(x,d_{n},2\epsilon)$. Since $\mu$ is  $\phi_1$-invariant for the TDS $(X,d_1,\phi_1)$,  by  Shannon-McMillan-Breiman theorem  we have 
\begin{align}\label{inequ 2.7}
\overline{h}_\mu^{BK}(\phi,d,2\epsilon)\leq \inf_{\diam(\alpha,d_1) \leq \epsilon}\limits h_\mu(\phi_1,\alpha).
\end{align}
Moreover, by Proposition \ref{prop 2.14} we have 
\begin{align}\label{inequ 2.8}
\inf_{\diam(\alpha,d_1) \leq \epsilon} h_{\mu}(\phi_1,\alpha) \leq  \underline{h}_{\mu}^K(\phi_1,d_1,\frac{\epsilon}{4})=\underline{h}_{\mu}^K(\phi,d,\frac{\epsilon}{4}).
\end{align}
By (\ref{inequ 2.6}), (\ref{inequ 2.7}) and  (\ref{inequ 2.8}), we  get
\begin{align}\label{inequ 2.9}
&\inf_{\diam(\alpha,d_1) \leq \epsilon} h_{\mu}(\phi_1,\alpha) \leq\underline{h}_{\mu}^K(\phi,d,\frac{\epsilon}{4})\leq \overline{h}_{\mu}^{BK}(\phi,d,2\epsilon)\leq \inf_{\diam(\alpha,d_1) \leq \frac{\epsilon}{16}} h_{\mu}(\phi_1,\alpha). 
\end{align}
Together with (\ref{inequ 2.6}), we complete the proof.
\end{proof}




\begin{proof}[Proof of Theorem \ref{thm 1.1}]
It follows from  the  Theorem \ref{thm 2.11} by  taking the corresponding  $\limsup_{\epsilon \to 0}$ and $\liminf_{\epsilon\to 0}$ for the known inequality.
\end{proof}

Actually, the divergent of Brin-Katok $\epsilon$-entropy can be  given by a more delicate manner based on the following fact.
\begin{prop}\label{prop 2.15}
If $f: \mathbb{R}_{>0} \rightarrow  \mathbb{R}_{>0}$ is a function satisfying $f(\epsilon)$ is non-decreasing when $\epsilon$ decreases, then 
\begin{align*}
	\limsup_{\epsilon \to 0}\frac{f(\epsilon)}{\logf}&= \limsup_{\mathbb{N} \ni n \to \infty}\frac{f(e^{-n})}{n},\\
	\liminf_{\epsilon \to 0}\frac{f(\epsilon)}{\logf}&= \liminf_{\mathbb{N} \ni n \to \infty}\frac{f(e^{-n})}{n}.
\end{align*}
\end{prop}

\begin{proof}
	Put $t=\logf$.  One has $\epsilon=e^{-t}$ and $t\to \infty$ as $\epsilon \to 0$. Then we have
	$$\limsup_{\epsilon \to 0}\frac{f(\epsilon)}{\logf}= \limsup_{ t \to \infty}\frac{f(e^{-t})}{t}= \limsup_{\mathbb{N} \ni n \to \infty}\frac{f(e^{-n})}{n},$$
	where the second equality follows from the  monotonicity of $f$, i.e, $f(e^{-t_1}) \geq f(e^{-t_2})$ if $t_1 \geq t_2$. The same proof works for the second equality  of the statement.
\end{proof}

\begin{cor}[A analogue of Brin-Katok entropy formula]
Let $(X,d,\phi)$ be a flow and $\mu \in \mathcal{E}_{\phi}(X)$.  Then for  $\mu$-a.e. $x\in X$, 
\begin{align*}
	\limsup_{\epsilon \to 0}\frac{1}{\logf}\overline{h}_\mu^{BK}(\phi,x,d,\epsilon)&=	\limsup_{\epsilon \to 0}\frac{1}{\logf}\underline{h}_\mu^{BK}(\phi,x,d,\epsilon)=\limsup_{\epsilon \to 0}\frac{1}{\logf}\inf_{\diam(\alpha,d_1) \leq \epsilon} h_{\mu}(\phi_1,\alpha),\\
	\liminf_{\epsilon \to 0}\frac{1}{\logf}\overline{h}_\mu^{BK}(\phi,x,d,\epsilon)&=	\liminf_{\epsilon \to 0}\frac{1}{\logf}\underline{h}_\mu^{BK}(\phi,x,d,\epsilon)=\liminf_{\epsilon \to 0}\frac{1}{\logf}\inf_{\diam(\alpha,d_1) \leq \epsilon} h_{\mu}(\phi_1,\alpha).
\end{align*}
\end{cor}

\begin{proof}
Let  $a:= \limsup_{\epsilon \to 0}\frac{1}{\logf}\inf_{\diam(\alpha,d_1) \leq \epsilon} h_{\mu}(\phi_1,\alpha)$. By Theorem \ref{thm 1.1} and  Proposition \ref{prop 2.15}, we have
$$\limsup_{n\to \infty}\frac{\overline{h}_{\mu}^{BK}(\phi,d,e^{-n})}{n}=\limsup_{n\to \infty}\frac{\underline{h}_{\mu}^{BK}(\phi,d,e^{-n})}{n}=a.$$

By virtue of Proposition \ref{pro 2.9}, there exists a Borel set $Z$ with  full  $\mu$-measure such that  for every $x\in Z$,
$$\limsup_{n\to \infty}\frac{\overline{h}_{\mu}^{BK}(\phi,x,d,e^{-n})}{n}=\limsup_{n\to \infty}\frac{\underline{h}_{\mu}^{BK}(\phi,x,d,e^{-n})}{n}=a.$$
Now fix $x\in Z$ and apply Proposition  \ref{prop 2.15} to the functions $f(\epsilon):=\overline{h}_\mu^{BK}(\phi,x,d,\epsilon)$ and $g(\epsilon):=\overline{h}_\mu^{BK}(\phi,x,d,\epsilon)$. It holds that
$$\limsup_{\epsilon \to 0}\frac{\overline{h}_{\mu}^{BK}(\phi,x,d,\epsilon)}{\logf}=\limsup_{\epsilon \to 0}\frac{\underline{h}_{\mu}^{BK}(\phi,x,d,\epsilon)}{\logf}=a.$$ 
This completes the proof.
\end{proof}

\subsection{Variational principles for metric mean dimension of flows}
In this subsection,  we   establish variational principles for metric mean dimension of flows, and prove Theorems \ref{thm 1.2} and \ref{thm 1.3}. 

The following variational principles for metric mean dimension of $\mathbb{Z}$-actions was established in \cite{gs20,shi}. 

\begin{lem} (\rm{cf.} \cite[Theorems 1.2 and 1.3]{ycz25}) \label{lem 2.16} 
Let $(X,d,T)$ be a TDS. Then  for every $h_{\mu}(d,\epsilon) \in \{\overline{h}_{\mu}^K(T,d,\epsilon), \inf_{\diam (\alpha,d) \leq \epsilon}\limits h_\mu(T,\alpha)\}$,
\begin{align*}
\overline{\rm mdim}_M(T,X,d)&=\limsup_{\epsilon \to 0}\frac{1}{\logf}\sup_{\mu \in \mathcal{E}_{T}(X)}h_{\mu}(d,\epsilon),\\
&=\limsup_{\epsilon \to 0}\frac{1}{\logf}\sup_{\mu \in \mathcal{M}_{T}(X)}h_{\mu}(d,\epsilon).
\end{align*}

 The variational principles also hold for $\underline{\rm mdim}_M(T,X,d)$ by changing $\limsup_{\epsilon \to 0}$ into $\liminf_{\epsilon \to 0}$.
\end{lem}

\begin{lem}\label{lem 2.17}
Let $(X,d,\phi)$ be a flow.
Given $\invariant$ and  a Borel  set  $B$ of $X$, for every $t\in [0,1]$  we define $\mu_t(B):=(\phi_{-t}\mu)(B)=\mu(\phi_t(B))$. Then  the function
\begin{align*}
f_B:t\in [0,1] \mapsto \mu_t(B)
\end{align*}
is Lebesgue measurable.
\end{lem}

\begin{proof}
We first show that for every open set $O\subset X$, the map $$f_O:t\in [0,1] \mapsto \mu_t(O)$$
is   Lebesgue measurable.

Clearly, $\mu_t\in \mathcal{M}_{\phi_1}(X)$ for every $t\in [0,1]$.
Notice that  $\phi: X \times [-1,1]\rightarrow X$ is uniformly continuous. Then  $\phi_{- t_n}x$ uniformly converges to $ \phi_{-t} x$  if $t_n\to t$ with $t_n\in [0,1]$. Let $f\in C(X,\mathbb{R})$ be a continuous potential on $X$. Then for every  $x\in X$, $f(\phi_{- t_n}x) \to f(\phi_{- t}x)$ as $n \to \infty$. Then, by Lebesgue's dominated convergence theorem, we have
\begin{align*}
\int fd \phi_{-t_n}\mu=\int  f(\phi_{- t_n}x)  d\mu\rightarrow\int  f(\phi_{- t}x)  d\mu=\int fd \phi_{-t}\mu,
\end{align*}
which yields that $\mu_{t_n}\to \mu_t$ whenever $t_n\rightarrow t$. Then $$t\in [0,1]\mapsto \mu_t\in \mathcal{M}_{\phi_1}(X)$$ is a continuous map. Now fix an open set $O$ of $X$. Using the equivalent characterization of the weak$^{*}$-topology on $\mathcal{M}(X)$ (cf. \cite[p.149]{w82}), we know that $\mu\in \mathcal{M}_{\phi_1}(X) \rightarrow \mu(O)$ is a lower semi-continuous function. Therefore, $f_O$ is lower semi-continuous, and hence is Lebesgue measurable.

 Put $$\mathcal{D}:=\{B\in \mathcal{B}(X): f_B~\text{is Lebesgue measurable}\}.$$
 Then $\mathcal{D}$ contains  the collection of all open sets of $X$. 
 We  verify that $\mathcal{D}$ is a $\lambda$-system:
 \begin{itemize}
 \item [(1)] clearly, $X\in \mathcal{D}$;
 \item [(2)] if $A, B\in \mathcal{D}$ and $A\subset B$, then
 $$f_{B\backslash A}(t)=\mu_t(B\backslash A)=f_B(t)-f_A(t),$$
 which implies that $B\backslash A\in \mathcal{D}$;
  \item [(3)] if $A_n\in \mathcal{D}$ and $A_n\subset A_{n+1}$ for every $n\geq 1$, then
  $$f_{\cup_{n\geq 1}A_n}(t)=\lim_{n\to \infty}\mu_t( A_n)=\lim_{n\to \infty}f_{A_n}(t).$$
This implies that $\cup_{n\geq 1}A_n \in \mathcal{D}$.  
 \end{itemize} 
 Thus, $\mathcal{D}$ is a $\lambda$-system.
Notice that  the collection of all open sets of $X$  generates the Borel $\sigma$-algebra $\mathcal{B}(X)$. The Dynkin's $\pi-\lambda$ Theorem\footnote[4]{It says that for any $\pi$-system $\mathcal{C}$ and $\lambda$-system $\mathcal{D}$ in s space $S$, if $\mathcal{C}\subset \mathcal{D}$, we have $\sigma(\mathcal{C})\subset \mathcal{D}$.} (cf.  \cite[Theorem 1.1]{kal21}) allows us to deduce that
$\mathcal{B}(X)\subset \mathcal{D}$. Consequently,  we have $ \mathcal{D}= \mathcal{B}(X)$. This means that for every $B\in \mathcal{B}(X)$, $f_B$ is Lebesgue measurable. 
\end{proof}

With the help of the above two lemmas, we  are ready to  prove Theorem \ref{thm 1.2}.
\begin{proof}[Proof of Theorem \ref{thm 1.2}]
It suffices to show  the variational principles hold for $\over$ since the similar arguments work for $\under$.

 We  divide the proof  into  three steps.

\emph{{Step 1:}} 
We   prove that  for every $h_{\mu}(\phi,d,\epsilon)\in  \mathcal{E}$,
\begin{align*}
\over&\leq \limsup_{\epsilon \to 0}\frac{1}{\logf}\sup_{\mu \in \mathcal{E}_{\phi}(X)}h_{\mu}(\phi,d,\epsilon),\\
&\leq \limsup_{\epsilon \to 0}\frac{1}{\logf}\sup_{\mu \in \mathcal{M}_{\phi}(X)}h_{\mu}(\phi,d,\epsilon).
\end{align*} 

To end \emph{step 1}, it suffices to show
\begin{align*}
\over&\leq \limsup_{\epsilon \to 0}\frac{1}{\logf}\sup_{\mu \in \mathcal{E}_{\phi}(X)}\overline{h}_{\mu}^K(\phi,d,\epsilon),\\
&\leq \limsup_{\epsilon \to 0}\frac{1}{\logf}\sup_{\mu \in \mathcal{M}_{\phi}(X)}\overline{h}_{\mu}^K(\phi,d,\epsilon).
\end{align*} 
 Then,  by   Theorem  \ref{thm 2.11},  the inequalities also  hold for other measure-theoretic   $\epsilon$-entropy  $ h_{\mu}(\phi,d,\epsilon) \in \mathcal{E} \backslash\{\overline{h}_{\mu}^K(\phi,d,\epsilon)\}$.

Fix  $\epsilon >0$ and $\delta \in (0,1)$. Let   $\mu$  be a $\phi_1$-ergodic Borel probability measure for the TDS $(X,d_1,\phi_1)$. For every $t\in [0,1]$, by Lemma \ref{lem 2.17},  we define $\mu_t(B):=\mu(\phi_t(B))$ and  
\begin{align*}
m(B):=\int_{0}^{1} \mu_t(B)dt,
\end{align*}
where  $B$  is a Borel  set of $X$. 
 Then  $\mu_t\in \mathcal{E}_{\phi_1}(X)$ and $m\in\mathcal{E}_{\phi}(X) $.  
 
 Given $n \in \mathbb{{N}}$, let $E_n$ be a subset of $X$ such  that  $\#E_n= R_{m}^{\delta}(\phi,d,n+1,\epsilon)$. Then 
$$ m(\cup_{x\in E_n} B(x,d_{n+1},\epsilon))=\int_0^1 \mu_t(\cup_{x\in E_n}B(x,d_{n+1},\epsilon))dt \geq 1-\delta,$$
where $d_{n+1}(x,y)=\max_{s\in [0,n+1]}d(\phi_sx,\phi_sy)$.
There exists $t_0\in [0,1]$ such that
\begin{align*}
1-\delta \leq \mu_{t_0}(\cup_{x\in E_n}B(x,d_{n+1},\epsilon))&=\mu(\cup_{x\in E_n} \phi_{t_0}B(x,d_{n+1},\epsilon))\\
&\leq \mu(\cup_{x\in E_n}B(\phi_{t_0}x,d_n,\epsilon))=\mu(\cup_{x\in E_n}B(\phi_{t_0}x,(d_1)_{n,\phi_1},\epsilon)) .\nonumber
\end{align*}
This implies that $R_{\mu}^{\delta}(\phi_1,d_1,n,\epsilon)\leq R_{m}^{\delta}(\phi,d, n+1,\epsilon)$ for every $n\in \mathbb{N}$, and hence
\begin{align*}
\overline{h}_{\mu}^K(\phi_1,d_1,\epsilon,\delta)\leq
\overline{h}_{m}^K(\phi,d,\epsilon,\delta)&\leq \sup_{m\in \mathcal{E}_{\phi}(X)}\overline{h}_{m}^K(\phi,d,\epsilon,\delta).
\end{align*}
Letting $\delta \to 0$, we  have
\begin{align}\label{inequ 2.11}
\overline{h}_{\mu}^K(\phi_1,d_1,\epsilon)\leq \sup_{m\in \mathcal{E}_{\phi}(X)}\overline{h}_{m}^K(\phi,d,\epsilon).
\end{align}
By Proposition \ref{prop 2.3} and Lemma \ref{lem 2.16}, we get
$$\overline{\rm mdim}_M (\phi, X, d )=\overline{\rm mdim}_M (\phi_1, X, d_1)=\limsup_{\epsilon \to 0}\frac{1}{\logf}\sup_{\mu \in \mathcal{E}_{\phi_1}(X)}\overline{h}_{\mu}^K(\phi_1,d_1,\epsilon).$$
Together with (\ref{inequ 2.11}), we get
\begin{align*}
\overline{\rm mdim}_M (\phi, X, d )\leq  \limsup_{\epsilon\to 0} \frac{1}{\log \frac{1}{\epsilon}} \sup_{\mu\in \mathcal{E}_{\phi}(X)} \overline{h}_{\mu}^K(\phi,d,\epsilon).
\end{align*}

\emph{{Step 2}:} 
We   show that  for every $h_{\mu}(\phi,d,\epsilon)\in  \mathcal{E}$,
\begin{align*}
\limsup_{\epsilon \to 0}\frac{1}{\logf}\sup_{\mu \in \mathcal{M}(X)}h_{\mu}(\phi,d,\epsilon) \leq \over.
\end{align*} 

To get the desired inequality, we  compare the definitions of  measure-theoretic $\epsilon$-entropy $h_{\mu}(\phi,d,\epsilon)$ and the $\epsilon$-topological entropy of $X$.  As for every  Katok's $\epsilon$-entropy  $h_{\mu}(\phi,d,\epsilon)\in  \mathcal{E}\backslash \{\overline{h}_{\mu}^{BK}(\phi,d,\epsilon)\}$,  it is clear that $h_{\mu}(\phi,d,\epsilon) \leq r(\phi,X,d,\epsilon)$  for every $\epsilon>0$. The remaining for us is to show  that for every $\epsilon >0$, 
$$\overline{h}_\mu^{BK}(\phi,d,\epsilon)\leq  s(\phi,X,d,\epsilon).$$

Fix $\epsilon >0$ and $\mu \in \mathcal{M}(X)$. Recall that $d_n(x,y)=\max_{s\in [0,n]} d(\phi_sx,\phi_sy)$ is the $n$-th Bowen metric on $X$. Let $0<s<\overline{h}_\mu^{BK}(\phi,d,\epsilon)$. Then there exists  a Borel set $E \subset X$ with positive $\mu$-measure    such that  for all   $x\in E$,
$$\limsup_{n\to \infty}\limits-\frac{\log \mu (B(x,d_n,\epsilon))}{n}>s.$$ 
We consider the set $$E_n=\{x\in E: \mu(B(x,d_n,\epsilon,\phi)) < e^{-ns}\}.$$ Then $E=\cup_{n\geq N}E_n$ for every $N\geq 1$. Hence,  one can choose $n\geq N$ (depending on $N$) such that $\mu(E_n)\geq \frac{1}{n(n+1)}\mu(E)$.  Assume that  $\{n_j\}_{j\geq 1}$ is a  resulting sequence.  Let  $F_{n_j} \subset E_{n_j} $ be a
$(d_{n_j},\epsilon)$-separated set  of $E_{n_j} $ with the  largest cardinality $s(E_{n_j} ,d_{n_j},\epsilon)$. Then
$$\mu(E_{n_j} )\leq \sum_{x\in F_{n_j} }\mu(B(x,d_{n_j},\epsilon))< s(E_{n_j} ,d_{n_j},\epsilon)\cdot e^{-{n_j} s}$$
Hence, we have $$s (X,d_{n_j},\epsilon)>\mu(E_{n_j} ) e^{{n_j} s} \geq \frac{e^{{n_j} s}}{n_j(n_j+1)}\mu(E) $$ for every $j\geq 1$, and hence $s\leq s(\phi,X,d,\epsilon)$. Letting $s \to \overline{h}_\mu^{BK}(\phi,d,\epsilon)$, we  get   $\overline{h}_\mu^{BK}(\phi,d,\epsilon)\leq  s(\phi,X,d,\epsilon)$, which yields that 
$\sup_{\mu\in \mathcal{M}(X)}  \overline{h}_\mu^{BK}(\phi,\epsilon) \leq s(\phi,X,d,\epsilon)$.

\emph{{Step 3:}} We prove  the following variational principles:
\begin{align*}
\over&=\limsup_{\epsilon \to 0}\frac{1}{\logf}\sup_{\mu \in \mathcal{M}_{\phi}(X)}\inf_{\diam (\alpha,d_1) \leq \epsilon}\limits h_\mu(\phi_1,\alpha),\\
&=\limsup_{\epsilon \to 0}\frac{1}{\logf}\sup_{\mu \in \mathcal{E}_{\phi}(X)}\inf_{\diam (\alpha,d_1) \leq \epsilon}\limits h_\mu(\phi_1,\alpha).
\end{align*}

It follows from Lemma \ref{lem 2.16} that
\begin{align}\label{ine 2.11}
\begin{split}
\over = \overline{\rm mdim}_M (\phi_1, X, d_1)&\geq \limsup_{\epsilon \to 0}\frac{1}{\logf}\sup_{\mu \in \mathcal{M}_{\phi}(X)}\inf_{\diam (\alpha,d_1) \leq \epsilon}h_\mu(\phi_1,\alpha),\\
& \geq \limsup_{\epsilon \to 0}\frac{1}{\logf}\sup_{\mu \in \mathcal{E}_{\phi}(X)}\inf_{\diam (\alpha,d_1) \leq \epsilon}h_\mu(\phi_1,\alpha).
\end{split}
\end{align} 

Fix $\epsilon>0$ and $\mu \in  \mathcal{E}_{\phi}(X)$. The inequality  (\ref{inequ 2.7}) shows that $$\overline{h}_\mu^{BK}(\phi,d,2\epsilon)\leq \inf_{\diam(\alpha,d_1) \leq \epsilon}\limits h_\mu(\phi_1,\alpha).$$
Using \emph{Step 1}, we get
\begin{align}\label{ine 2,12}
\over\leq  \limsup_{\epsilon \to 0}\frac{1}{\logf}\sup_{\mu \in \mathcal{E}_{\phi}(X)}\inf_{\diam (\alpha,d_1) \leq \epsilon}h_\mu(\phi_1,\alpha).
\end{align}
We finish \emph{Step 3} by (\ref{ine 2.11}) and (\ref{ine 2,12}).
\end{proof}


Next we define measure-theoretic metric mean dimension  and prove  Theorem \ref{thm 1.3}.  
 
By  \cite[Proposition 3.3.26]{fh20}, every $\phi$-ergodic measure is an extreme point of $\mathcal{M}_{\phi}(X)$. The well-known Krein-Milman theorem says that  the closure of the  convex hull $\mathrm{co}(\mathcal{E}_{\phi}(X))$ of $\phi$-ergodic measures  is $\mathcal{M}_{\phi}(X)$, i.e.,
$$\mathcal{M}_{\phi}(X)=\overline{\{\sum_{j=1}^k p_j \mu_j:k\geq 1,~\text{for}~1\leq j\leq k,~ \mu_j \in \mathcal{E}_{\phi}(X),~0\leq p_j \leq 1,~ \sum_{j=1}^k p_j=1\}}.$$ 

In Theorem  \ref{thm 1.2}, the variational principles of  metric mean dimension  hold for several   measure-theoretic $\epsilon$-entropies without a unified form.  To get the classical-type variational principles for metric mean dimension of flows,  inspired by  the definition of measure-theoretic metric mean dimension of invariant measures of $\mathbb{Z}$-actions \cite[Definition 3.26]{yz25}, we  introduce the following: 

\begin{df}\label{df 2.18}
	$(a)$ Fix $\epsilon  >0$.
	Let $\mathcal{E}’=\mathcal{E}\cup \{\inf_{\diam (\alpha,d_1) \leq \epsilon}\limits h_\mu(\phi_1,\alpha)\}$. For $\mu \in  \mathrm{co}(\mathcal{E}_{\phi}(X))$ and $h_{\mu}(\phi,d,\epsilon) \in  \mathcal{E}'$, we rewrite $\mu=\sum_{j=1}^kp_j \mu_j,~\mu_j \in \mathcal{E}_{\phi}(X),~ \sum_{j=1}^{k}p_j=1,~0\leq p_j\leq 1,~j=1,...,k$. The measure-theoretic $\epsilon$-entropy of $\mu$ is defined by 
	\begin{align*}
	\Hat{h}_{\mu}(\phi,d,\epsilon):=\sum_{j=1}^{k}p_jh_{\mu_j}(\phi,d,\epsilon).
	\end{align*}

	$(b)$ For $\mu \in  \mathcal{M}_{\phi}(X)$, let $\mathcal{M}_{\phi}(\mu)$  denote the collection of all families $\{\mu_{\epsilon}\}_{\epsilon >0} \subset \mathrm{co}(\mathcal{E}_{\phi}(X))$  that converge to $\mu$  as $\epsilon \to 0$ in the weak$^{*}$-topology. 
	The upper and lower  measure-theoretic metric mean dimension of $\mu$ are respectively defined by
	\begin{align*}
		\overline{\rm mdim}_\mu(\phi,d):&=\sup_{(\mu_{\epsilon})\in \mathcal{M}_{\phi}(\mu)} \{\limsup_{\epsilon \to 0} \frac{\Hat{h}_{\mu_{\epsilon}}(\phi,d,\epsilon)}{\logf}\},\\
		\underline{\rm mdim}_\mu(\phi,d):&=\sup_{(\mu_{\epsilon})\in \mathcal{M}_{\phi}(\mu)} \{\liminf_{\epsilon \to 0} \frac{\Hat{h}_{\mu_{\epsilon}}(\phi,d,\epsilon)}{\logf}\}.
	\end{align*}
\end{df}

 
 \begin{rem}
 By Theorem \ref{thm 2.11}, for every  $h_{\mu}(\phi,d,\epsilon) \in \mathcal{E}$, there exist  constants $C_1,C_2>0$ such that for all $\epsilon>0$,
 $$\inf_{\diam(\alpha,d_1) \leq C_1\epsilon} h_{\mu}(\phi_1,\alpha) \leq h_{\mu}(\phi,d,\epsilon)\leq \inf_{\diam(\alpha,d_1) \leq C_2\epsilon} h_{\mu}(\phi_1,\alpha).$$
The same holds for $\Hat{h}_{\mu}(\phi,d,\epsilon)$. Therefore, Definition \ref{df 2.18} is independent of the choice of the candidates in $\mathcal{E}'$. Before we calculate the measure-theoretic mean dimension of invariant measures, we will verify which candidate $h_{\mu}(\phi,d,\epsilon)$ in $\mathcal{E}'$ to use.
 \end{rem}



\begin{lem} \cite[Lemma 4.2]{yz25}\label{lem 2.18}
Let  $(X,d,T)$ be a TDS. Then for every  $\epsilon  >0$, there exists $\mu_{\epsilon} \in \mathcal{E}_{T}(X)$  such that
\begin{align*}
r(T,X,d,\epsilon)
\leq \inf_{\diam (\alpha, d)\leq \frac{\epsilon}{64}}h_{\mu_{\epsilon}}(T,\alpha).
\end{align*}
\end{lem}
Now we  are  ready  to  prove Theorem \ref{thm 1.3}.

\begin{proof}[Proof of Theorem \ref{thm 1.3}]

Fix $\mu  \in \mathcal{M}_{\phi}(X)$ and let $h_{\mu}(\phi,d,\epsilon):= \overline{h}_{\mu}^K(\phi,d,\epsilon)$. Since $  \overline{h}_{\mu}^K(\phi,d,\epsilon) \leq  r(\phi,X,d,\epsilon)$ for every $\mu  \in  \mathcal{E}_{\phi}(X)$, after taking the finite convex combinations  it holds that $h_{\mu}(\phi,d,\epsilon) \leq  r(\phi,X,d,\epsilon) $ for every $\mu \in  \mathrm{co}(\mathcal{E}_{\phi}(X))$. Using this fact, we get $$\limsup_{\epsilon \to 0} \frac{\Hat{h}_{\mu_{\epsilon}}(\phi,d,\epsilon)}{\logf}\leq \over.$$
for every $(\mu_{\epsilon}) \in  \mathcal{M}_{\phi}(\mu)$. This yields  that  $$\sup_{\mu \in \mathcal{M}_{\phi}(X)}\overline{\rm mdim}_\mu(\phi,d)\leq \over.$$

On the other hand, taking $h_{\mu}(\phi,d,\epsilon):= \overline{h}_{\mu}^K(\phi,d,\epsilon)$ again and applying Lemma \ref{lem 2.18}  to $(X,d_1,\phi_1)$, for every $\epsilon > 0$ there exists $\mu_{\epsilon} \in \mathcal{E}_{\phi_1}(X)$ such that $$r(\phi_1,X,d_1,\epsilon)\leq \inf_{\diam (\alpha, d_1)\leq \frac{\epsilon}{64}}h_{\mu_{\epsilon}}(\phi_1,\alpha)\leq \overline{h}_{\mu_{\epsilon}}^K(\phi_1,d_1,\frac{\epsilon}{256}),$$
where the second inequality holds by   Proposition  \ref{prop 2.14}. Using (\ref{inequ 2.11}),  for every $\epsilon>0$ there exists $m_{\epsilon}\in  \mathcal{E}_{\phi}(X)$ such that  \begin{align}\label{ienqu 2.11}
r(\phi,X,d,\epsilon)=r(\phi_1,X,d_1,\epsilon) \leq \overline{h}_{\mu_{\epsilon}}^K(\phi_1,d_1,\frac{\epsilon}{256}) \leq \overline{h}_{{m_\epsilon}}^K(\phi,d,\frac{\epsilon}{256})=\hat{h}_{m_{\epsilon}}(\phi,d,\frac{\epsilon}{256}),
\end{align}
where  the last equality holds since $m_{\epsilon}$ is $\phi$-ergodic. Therefore, there exists $c_1>0$  such that for all $\epsilon >0$,
\begin{align}\label{inequ 2.14}
r(\phi,X,d,c_1\epsilon)\leq \hat{h}_{m_{\epsilon}^{'}}(\phi,d,\epsilon)
\end{align}
for some $m_{\epsilon}^{'} \in \mathcal{E}_{\phi}(X) $. Choose  a strictly decreasing sequence $\{\epsilon_n\}$ of positive real numbers  such  that $\epsilon_n \to 0$ as $n \to \infty$ and
$$\over=\lim_{n \to \infty}\frac{1}{\log \frac{1}{\epsilon_n}}r(\phi,X,d,c_1\epsilon).$$
  By   the compactness of  $\mathcal{M}_{\phi}(X)$, we may assume that for some $m_0 \in  \mathcal{M}_{\phi}(X)$, $m_{\epsilon_n}^{'}  \to m_0$ as $n \to \infty$ in the weak$^{*}$-topology. For every $\epsilon \in (\epsilon_{n+1}, \epsilon_n]$, we define  $m_{\epsilon}^{''}= m_{\epsilon_n}^{'}$. Then  $(m_{\epsilon}^{''})_{\epsilon} \in \mathcal{M}_{\phi}(m_0) $. By (\ref{inequ 2.14}), we have
$$\over\leq \limsup_{n\to \infty} \frac{\Hat{h}_{m_{\epsilon_n}^{'}}(\phi,d,\epsilon_n)}{\log \frac{1}{\epsilon_n}}\leq \limsup_{\epsilon \to 0} \frac{\Hat{h}_{m_{\epsilon}^{''}}(\phi,d,\epsilon)}{\logf},$$
and hence
\begin{align}\label{inequ 2.13}
\over \leq \overline{\rm mdim}_{m_0}(\phi,d), 
\end{align}
and $m_0$ is the $\phi$-invariant measure that attains  the supremum of the variational principle.
\end{proof}

\begin{rem}
It is unclear that  whether $\under=\max_{\mu \in \mathcal{M}_{\phi}(X)}\limits \underline{\rm mdim}_\mu(\phi,d)$  holds for every flow $(X,d,\phi)$.
\end{rem}

\section{A Hausdorff-type metric mean dimension of flows}

In this section, we give  the dimensional characterization for metric mean dimension of  flows, and prove Theorems \ref{thm 1.5} and  \ref{thm 1.4}.

\subsection{Bowen metric mean dimension of  flows on subsets} 

Inspired by  the definition of Bowen topological entropy of non-compact subsets \cite{b73}, we   extend the concept of metric mean dimension to any subset of $X$ using Carath\'eodory-Pesin structure \cite{p97}.

Let $(X,d,\phi)$ be a flow. Given a non-empty subset $Z\subset X$, $\epsilon >0, N\in\mathbb{N}$, and $s\geq 0$, we put 
$$M_{N,\epsilon}^{s}(\phi,Z,d)=\inf\{\sum_{i\in I}\limits  e^{-n_i s}\},$$
where the infimum  ranges over all  finite or countable   covers $\{B(x_i, d_{n_i},\epsilon)\}_{i\in I}$ of $Z$ with $x_i\in X$ and  $n_i \geq N.$

Since $M_{N,\epsilon}^{s}(\phi,Z,d)$ is non-decreasing when $N$ increases,  the limit
$$M_{\epsilon}^s(\phi,Z,d):=\lim\limits_{N\to \infty}M_{N,\epsilon}^{s}(\phi,Z,d)$$ exists.  One can check that there is a critical value  of the parameter $s$ such  that $M_{\epsilon}^s(\phi,Z,d)$  jumps from $\infty$ to $0$, i.e.,
\begin{align*}
	h_{top}^B(\phi,Z,d,\epsilon)&:=\inf\{s>0: M_{\epsilon}^s(\phi,Z,d)=0\}\\
	&=\sup\{s>0: M_{\epsilon}^s(\phi,Z,d)=\infty\}.
\end{align*}
\begin{df}\label{df 3.2}
The  Bowen upper and lower metric mean dimensions of $\phi$  on  $Z$ are respectively defined by  
\begin{align*}
\overline{\rm mdim}_M^B(\phi,Z,d)&=\limsup_{\epsilon \to 0}\frac{h_{top}^B(\phi,Z,d,\epsilon)}{\log \frac{1}{\epsilon}},\\
\underline{\rm mdim}_M^B(\phi,Z,d)&=\liminf_{\epsilon \to 0}\frac{h_{top}^B(\phi,Z,d,\epsilon)}{\log \frac{1}{\epsilon}}.
\end{align*}
\end{df}

If $Z=\emptyset$, we let $\overline{\rm mdim}_M^B(\phi,Z,d)=\underline{\rm mdim}_M^B(\phi,Z,d):=0$.
The elementary properties of Bowen metric mean dimension are summarized as follows:

\begin{prop} \label{prop 3.2}
Let  $(X,d,\phi)$ be a  flow. Then 
\begin{itemize}
\item [(1)] if $Z_1\subset Z_2 \subset X$, then  $\overline{\rm mdim}_M^B(\phi,Z_1,d) \leq  \overline{\rm mdim}_M^B(\phi,Z_2,d)$ and $\underline{\rm mdim}_M^B(\phi,Z_1,d) \leq  \underline{\rm mdim}_M^B(\phi,Z_2,d)$;  
\item [(2)] if $Z=\cup_{j=1}^{\infty}Z_j$ is a countable union of $Z_j$, then $h_{top}^B(\phi,Z,d,\epsilon)=\sup_{1\leq j <\infty}\limits h_{top}^B(\phi,Z_j,d,\epsilon)$ for every $\epsilon>0$;
\item [(3)] if $Z=\cup_{j=1}^{m}Z_j$ is a finite union of $Z_j$, then $\overline{\rm mdim}_M^B(\phi,Z,d) =\max_{1\leq j\leq m}\limits \overline{\rm mdim}_M^B(\phi,Z_j,d)$ and $\underline{\rm mdim}_M^B(\phi,Z,d) \geq \max_{1\leq j\leq m}\limits \underline{\rm mdim}_M^B(\phi,Z_j,d)$ ;
\item [(4)] if $Z\subset X$  is a $\phi$-invariant compact subset, then $\overline{\rm {mdim}}_{M}(\phi,Z,d)=\overline{\rm mdim}_M^B(\phi,Z,d)$ and  $\underline{\rm {mdim}}_{M}(\phi,Z,d)=\underline{\rm mdim}_M^B(\phi,Z,d)$.
\end{itemize}
\end{prop}

\begin{proof}
The statements (1)-(3) essentially follow from the definition of Bowen metric mean dimension.

(4).
For every   $(n,\epsilon)$-spanning set $E$ of $Z$ with the  minimal cardinality,   we have  $Z \subset \cup_{x\in E}B(x,d_n,\epsilon)$. Using this fact,  one can obtain
the inequality $$\overline{\rm mdim}_M^B(\phi,Z,d)\leq \overline{\rm {mdim}}_{M}(\phi,Z,d).$$  

Fix $\epsilon>0$ and let  $s>h_{top}^B(\phi,Z,d,\epsilon)$. The  compactness of $Z$ ensures that there exist  $N_0 \in \mathbb{N}$  and a finite open cover $\{B(x_i,d_{t_i},\epsilon)\}_{i\in I}$ of $Z$  with $t_i \geq N_0$ such that
$\sum_{i\in I} {e}^{-st_i}<1,$ where $I$ is finite index set. This means that
\begin{align*}
\sum_{k=0}^{\infty}\sum_{j_0,...,j_k\in I}e^{-s(t_{j_0}+\cdots+ t_{j_k})}= \sum_{k=0}^{\infty}(\sum_{i\in  I}e^{-st_i})^{k+1} <\infty.
\end{align*}

Let $M=\max_{i\in I} t_i$. We  define $c_0:=0$ and $c_k:=\sum_{i=0}^{k-1}t_{j_i}$ for every $k \geq 1$, where $j_i\in I $.  For every sufficiently large  $N \geq N_0$, we consider the following family consisting of open sets:
\begin{align*}
\mathcal{F}_N:=\{\mathop\cap_{i=0}^{k}\limits \phi_{-c_i}B(x_{{j_i}},d_{t_{j_i}},\epsilon): k\geq 0, 
N\leq  \sum_{i=0}^{k}t_{j_i}<N+M,j_i\in I\}.
\end{align*} 
Without loss of generality,  we  assume that each element in $ \mathcal{F}_N$ is not empty. Then for every $A\in \mathcal{F}_N$, we can choose  a point $x_A\in A $ such that  $A\subset B(x_{A},d_N, 2\epsilon)$.  Notice that $\{B(x_i,d_{t_i},\epsilon)\}_{i\in I}$ is an open cover of $Z$. Then 
$Z\subset \mathop\cup_{A\in \mathcal{F}_N}\limits B(x_{A},d_N,2\epsilon).$  Thus, we have
\begin{align*}
r(\phi,Z,d_N,2\epsilon)\cdot e^{-sN}&\leq \# \mathcal{F}_N\cdot e^{-sN}\\
&\leq e^{sM} \sum_{k\geq0}~\sum_{N\leq t_{j_0}+...+t_{j_k}<N+M} e^{-s(t_{j_0}+\cdots+ t_{j_k})}   <\infty,
\end{align*}
where the right hand side of the last inequality is  a constant independent of $N$.
This implies that $r(\phi,Z,d,2\epsilon)\leq s$, and hence $$r(\phi,Z,d,2\epsilon)\leq h_{top}^B(\phi,Z,d,\epsilon)$$ by letting $s \to h_{top}^B(\phi,Z,d,\epsilon)$. This shows  the reverse inequality $$\overline{\rm mdim}_M^B(\phi,Z,d)\geq  \overline{\rm {mdim}}_{M}(\phi,Z,d).$$
\end{proof}

\subsection{Inverse variational principle for Kolmogorov-Sinai $\epsilon$-entropy}

We  first compare the  divergent rates of  lower Brin-Katok $\epsilon$-entropy $\underline{h}_\mu^{BK}(\phi,d,\epsilon)$  and $\epsilon$-Bowen  topological entropy.

\begin{thm}\label{thm 3.3}
Let $(X,d,\phi)$ be a flow and $\mu \in \mathcal{E}_{\phi}(X)$.  Then
\begin{align*}
\limsup_{\epsilon \to 0}\frac{1}{\logf}\underline{h}_\mu^{BK}(\phi,d,\epsilon)&=\limsup_{\epsilon \to 0}\frac{1}{\logf}\inf_{\text{Borel set Z of $X$ with}~\mu(Z)=1} h_{top}^B(\phi,Z,d,\epsilon),\\
\liminf_{\epsilon \to 0}\frac{1}{\logf}\underline{h}_\mu^{BK}(\phi,d,\epsilon)&=\liminf_{\epsilon \to 0}\frac{1}{\logf}\inf_{\text{Borel set  $Z$ of $X$ with}~\mu(Z)=1} h_{top}^B(\phi,Z,d,\epsilon),
\end{align*}
where the infimum ranges over  all Borel sets  $Z$ of $X$ with full $\mu$-measure.
\end{thm}


 \begin{lem}[A 3$r$-covering lemma]\label{lem 3.4}
Let $(X,d)$ be a compact metric space. Let $r>0$ and $\mathcal{B}(r)=\{B(x,d_n,r):x\in X, n\geq 1\}$. Then for each family $\mathcal{F}\subset \mathcal{B}(r)$, there exists a subfamily $\mathcal{G}\subset \mathcal{F}$ consisting  of pairwise disjoint balls  such that
 $$\cup_{B\in \mathcal{F}}B\subset \cup_{ B(x,d_n,r)\in \mathcal{G}}~B(x,d_n,3r).$$
 \end{lem} 
 
 \begin{proof}
 It slightly modifies the arguments of the proof  of \cite[Lemma 1]{mw08},
 \end{proof}
 
\begin{lem}\label{lem 3.5}
Let $(X,d,\phi)$ be a flow. Let $\mu \in \mathcal{M}(X)$ and $\epsilon >0$. Assume that $0\leq\alpha<\infty$ and  $Z$ is a Borel set with $\mu(Z)>0$ such that $\underline{h}_\mu^{BK}(\phi,x,d,\epsilon)\geq \alpha$ for all $x\in Z$, then  $$h_{top}^B(\phi,Z,d,\frac{\epsilon}{2})\geq \alpha.$$ 

Consequently,  for every  Borel set $Z\subset X$   with $\mu(Z)=1$, we have  $$\underline{h}_\mu^{BK}(\phi,d,\epsilon)\leq h_{top}^B(\phi,Z,d,\frac{\epsilon}{2}).$$
 \end{lem}
 \begin{proof}
 It is clear that $h_{top}^B(\phi,Z,d,\frac{\epsilon}{2})\geq 0$ whenever $\alpha=0$. Now we assume  $0<\alpha<\infty$, and let $0<s<\alpha$. Then $\underline{h}_\mu^{BK}(\phi,x,d,\epsilon)> s$ for all $x\in Z$. By  decomposing $Z$ into  a countably many union of  Borel subsets of $Z$, there is  a Borel set $E\subset Z$ with $\mu(E)>0$  and  a positive  integer $N $ such that  $$\mu(B(x,d_{n},\epsilon))<e^{-ns} $$for any $x\in E$ and $n\geq N$.  Let $\{B(x_i,d_{n_i},\frac{\epsilon}{2})\}_{i\in I}$ be a   finite  or countable  cover of $E$  with $n_i \geq N$. We may  assume that  $B(x_i,d_{n_i},\frac{\epsilon}{2})\cap E\not= \emptyset$ for every $i\in I$. 
Choose $y_i \in  B(x_i,d_{n_i},\frac{\epsilon}{2})\cap E$ for each $i \in I$ so that  $\cup_{i\in I}B(y_i,d_{n_i},\epsilon)\supseteq E$. This yields that 
\begin{align*}
\sum_{i\in  I}e^{-sn_i}\geq \sum_{i\in  I}\mu(B(y_i,d_{n_i},\epsilon))\geq \mu(E)>0.
\end{align*}
Then  $M_{\frac{\epsilon}{2}}^s(\phi,Z,d)\geq M_{N,\frac{\epsilon}{2}}^s(\phi,Z,d)>0$. This means that $ h_{top}^B(\phi,Z,d,\frac{\epsilon}{2})\geq h_{top}^B(\phi,E,d,\frac{\epsilon}{2})\geq s$. Letting $s \to \alpha$, we obtain that $ h_{top}^B(\phi,Z,d,\frac{\epsilon}{2}) \geq \alpha$.

Now assume that $Z$ is  a Borel set with $\mu(Z)=1$. Then $\underline{h}_\mu^{BK}(\phi,d,\epsilon)=\int_Z \underline{h}_\mu^{BK}(\phi,x,d,\epsilon) d\mu$. Let $0<t<\underline{h}_\mu^{BK}(\phi,d,\epsilon)$. There exists a Borel subset $E\subset Z$ with $\mu(E)>0$  satisfying $\underline{h}_\mu^{BK}(\phi,x,d,\epsilon)\geq t$ for all $x\in E$.
Then  $$h_{top}^B(\phi,Z,d,\frac{\epsilon}{2})\geq h_{top}^B(\phi,E,d,\frac{\epsilon}{2})\geq t.$$ We  complete the proof by letting $t\to \underline{h}_\mu^{BK}(\phi,d,\epsilon) $.
 \end{proof}

With the help of Lemmas \ref{lem 3.4} and \ref{lem 3.5}, we  prove Theorem \ref{thm 3.3}.

\begin{proof}[Proof of Theorem \ref{thm 3.3}]
Fix $\epsilon >0$. By Lemma \ref{lem 3.5}, we have 
\begin{align}\label{inequ 3.1}
\underline{h}_\mu^{BK}(\phi,d,\epsilon) \leq \inf_{\text{Borel set $Z$ of $X$ with}~\mu(Z)=1} h_{top}^B(\phi,Z,d,\frac{\epsilon}{2}).
\end{align}

Now we define a set (depending on $\epsilon$):
$$E:=\{x\in X: \underline{h}_\mu^{BK}(\phi,x,d,\epsilon)=\underline{h}_\mu^{BK}(\phi,d,\epsilon)\}.$$
By Proposition \ref{pro 2.9}, one has  $\mu(E)=1$ since $\mu $ is ergodic.  Let $s>\underline{h}_{\mu}^{BK}(\phi,d,\epsilon)$. For each  $x\in E$,  we  choose a  strictly  increasing   sequence $\{n_j(x)\}_j$  of   positive integers  satisfying $$\mu(B(x,d_{{n_j}(x)},\epsilon))>e^{-{n_j}(x)s}>0.$$ For every $N\geq 1$,  the family $$\mathcal{F}_N:= \{B(x,d_{{n_j}(x)},\epsilon): x\in E,\mu(B(x,d_{{n_j}(x)},\epsilon))>e^{-n_j(x)s},n_j(x)\geq N\}$$ is a cover of $E$. By Lemma \ref{lem 3.4},  there  exists  a (countable) subfamily  $\{B(x_i, d_{n_i},\epsilon)\}_{i\in I}$  of pairwise disjoint Bowen balls   such that $E\subset \cup_{i\in I}B(x_i,d_{{n_i}},3\epsilon).$ 
It follows that
\begin{align*}
M_{N,3\epsilon}^s(\phi,E,d)\leq \sum_{i\in I}e^{-n_is}<\sum_{i\in I}\mu(B(x_i,d_{{n_i}},\epsilon))\leq 1,
\end{align*}
which implies that  $M_{3\epsilon}^s(\phi,E,d)\leq 1$. Letting $s \to \underline{h}_{\mu}^{BK}(T,\epsilon)$ gives us 
\begin{align}\label{inequ 3.2}
\inf_{\text{Borel set $Z$ of $X$ with}~\mu(Z)=1} h_{top}^B(\phi,Z,d,3\epsilon)\leq h_{top}^B(\phi,E,d,3\epsilon)\leq  \underline{h}_{\mu}^{BK}(\phi,d,\epsilon).
\end{align}
This completes the proof by (\ref{inequ 3.1}) and (\ref{inequ 3.2}).
\end{proof}



Let $\mu \in \mathcal{E}_{\phi}(X)$, and let
$$G_{\mu}:=\{x\in X: \lim_{t\to \infty}\frac{1}{t}\int_{0}^{t}f(\phi_{\tau}x)d\tau= \int f d \mu~\text{for all}~f\in C(X,\mathbb{R})\}$$
denote the set of  generic points of $\mu$.  From the   measure-theoretic viewpoint, the Birkhoff ergodic theorem of  flows (cf. \cite[Theorem 3.2.17]{fh20}) implies  that $\mu(G_{\mu})=1$. For flows with infinite  measure-theoretic entropy, we supply a refined result stating that  the divergent rate of  $\epsilon$-Bowen topology entropy of $G_{\mu}$ coincides with the one of  Kolmogorov-Sinal $\epsilon$-entropy.

\begin{thm}\label{thm 3.8}
	Let $(X,d,\phi)$ be a flow and let $\mu \in \mathcal{E}_{\phi}(X)$. Then  for every   Borel set $E$ with full $\mu$-measure,
	\begin{align*}
		\overline{\rm mdim}_M^B(\phi,G_{\mu}\cap E,d)=	\overline{\rm mdim}_M^B(\phi,G_{\mu},d)&= \limsup_{\epsilon \to 0}\frac{1}{\logf}\inf_{\diam(\alpha,d_1) \leq \epsilon} h_{\mu}(\phi_1,\alpha),\\
			\underline{\rm mdim}_M^B(\phi,G_{\mu}\cap E,d)=\underline{\rm mdim}_M^B(\phi,G_{\mu},d)&= \liminf_{\epsilon \to 0}\frac{1}{\logf}\inf_{\diam(\alpha,d_1) \leq \epsilon} h_{\mu}(\phi_1,\alpha).
	\end{align*}
\end{thm}

Given  $\mu \in \mathcal{M}(X)$ and $\epsilon>0$,   Pfister-Sullivan's $\epsilon$-entropy of $\mu$ is defined by
$$PS_\mu(\phi,d,\epsilon)=\inf_{F\ni \mu}\limsup_{t \to \infty}\frac{1}{t}\log s(X_{t,F},d_t,\epsilon),$$
where $X_{t,F}=\{x\in X: \frac{1}{t} \int_{0}^t \delta_{\phi_{\tau}(x)} d\tau \in F\}$, and  the infimum ranges over all neighborhoods  $F$  in $\mathcal{M}(X)$ of $\mu$.

\begin{lem} \cite[Lemma 3.2]{ycz25} \label{lem 3.9}
	Let $(X,d,T)$ be a TDS and $\mu \in \mathcal{E}_T(X)$. Then for every $\epsilon >0$,
	$$PS_\mu(T,d,\epsilon)\leq  \overline{h}_\mu^{BK}(T,d,\frac{\epsilon}{12}).$$
\end{lem}

\begin{proof}[Proof of Theorem \ref{thm 3.8}]

	It directly follows from Lemma \ref{lem 3.5} that for every $\epsilon>0$, we have 
	\begin{align}\label{inequ 3.3}
		\underline{h}_\mu^{BK}(\phi,d,\epsilon)\leq h_{top}^B(\phi,G_{\mu}\cap E,d,\frac{\epsilon}{2})\leq \leq h_{top}^B(\phi,G_{\mu},d,\frac{\epsilon}{2}).
	\end{align} 
	
	Fix  $\epsilon >0$. Notice that for the flow $(X,d,\phi)$ and its discrete sample $(X,d_\tau,\phi_\tau)$, it holds that  for every $\tau >0$,
	\begin{align*}
		PS_\mu(\phi,d,\epsilon)&= \frac{PS_\mu(\phi_\tau,d_\tau,\epsilon)}{\tau},\\
		\overline{h}_\mu^{BK}(\phi,d,\epsilon)&= \frac{\overline{h}_\mu^{BK}(\phi_\tau,d_\tau,\epsilon)}{\tau}
	\end{align*} 
	This fact allows us to   assume that $\mu$ is  ergodic for the TDS $(\phi_1,d_1,X)$. 
	Then,  the  Lemma \ref{lem 3.9} implies that   $PS_\mu(\phi,d,\epsilon)\leq  \overline{h}_\mu^{BK}(\phi,d,\frac{\epsilon}{12})$. Hence, the left is to bound $h_{top}^B(\phi,G_{\mu},d,\epsilon)$  above by  $PS_\mu(\phi,d,\epsilon)$.
	
	Let $\gamma >0$. Choose a neighborhood $F$ of $\mu$ and $N_0\in \mathbb{N}$ such that $$s(X_{n,F},d_n,\epsilon)< e^{n(PS_\mu(\phi,d,\epsilon) +\gamma)}$$ 
	for all $n \geq N_0$. Then for every $N\geq N_0$, we have $G_{\mu}\subset \cup_{n\geq N}X_{n,F}$ by comparing the definitions.  Let $s=PS_\mu(\phi,d,\epsilon) +2\gamma$. Then 
	$$M_{N,\epsilon}^{s}(\phi,G_{\mu},d)\leq \sum_{n\geq N} M_{N,\epsilon}^{s}(\phi,X_{n,F},d) \leq \sum_{n\geq N}e^{n(PS_\mu(\phi,d,\epsilon) +\gamma)}\cdot e^{-ns}\leq \sum_{n\geq N}e^{-\gamma n},$$
	which yields that $M_{\epsilon}^{s}(\phi,G_{\mu},d)=0$ by letting $N \to \infty$ and hence $h_{top}^B(\phi,G_{\mu},d,\epsilon) \leq  PS_\mu(\phi,d,\epsilon) +2\gamma$. Letting  $\gamma \to 0$, we deduce that
	\begin{align}\label{inequ 3.4}
		h_{top}^B(\phi,G_{\mu},d,\epsilon)\leq PS_\mu(\phi,d,\epsilon) \leq  \overline{h}_\mu^{BK}(\phi,d,\frac{\epsilon}{12}).
	\end{align}
	Combining  the inequalities (\ref{inequ 3.3}), (\ref{inequ 3.4}) and Theorem \ref{thm 1.1}, we  complete the proof.
\end{proof}

\begin{proof}[Proof of Theorem \ref{thm 1.5}]
It is a immediate consequence of Theorems  \ref{thm 1.1}, \ref{thm 3.3} and \ref{thm 3.8}.
\end{proof}
\subsection{Variational principle for Bowen metric mean dimension on compact subsets}
In this subsection, we prove Theorem \ref{thm 1.4}.

To establish the variational principle for Bowen metric mean dimension on compact subsets,  whose form is close to Theorem \ref{thm 1.3},  using lower Brin-Katok local $\epsilon$-entropy we  introduce mean Brin-Katok entropy of   probability measures.

Let $(X,d,\phi)$ be a flow.  By the weak$^{*}$-topology of $\mathcal{M}(X)$, the set $$\mathcal{M}_K(X):=\{\mu\in \mathcal{M}(X): \mu(K)=1\}$$ is  compact. Let $\mathcal{M}_{K}(\mu)=\{(\mu_\epsilon)_{\epsilon >0}: \mu_\epsilon \in  \mathcal{M}_K(X), \mu_{\epsilon} \to \mu~\text{as}~\epsilon \to 0\}$.

\begin{df}
 For $\mu \in \mathcal{M}_K(X)$,  we   define the upper Brin-Katok  entropy of $\mu$ as
\begin{align*}
\overline{\rm mdim}_{\mu}^{BK}(\phi,d)&=\sup_{(\mu_\epsilon)\in \mathcal{M}_{K}(\mu)}\{\limsup_{\epsilon \to 0}\frac{\underline{h}_{\mu_\epsilon}^{{BK}}(\phi, d, \epsilon)}{\logf} \}.
\end{align*}
\end{df}

In appendix,  without requiring the dynamics, just using  a family of compatible metrics on $X$  we  establish  variational principle for Bowen metric mean dimension of non-empty  compact subsets (see Theorem \ref{thm 5.5}). 

\begin{proof}[Proof of Theorem \ref{thm 1.4}]
It  is a direct consequence of Theorem \ref{thm 5.5}.
\end{proof}

\begin{cor}
Let $(X,d,\phi)$ be a flow. Then 
\begin{align*}
\overline{\rm mdim}_M(\phi,X,d)&=\max_{\mu \in \mathcal{M}(X)}\overline{\rm mdim}_{\mu}^{BK}(\phi,d).
\end{align*}
\end{cor}

\subsection{Bowen metric mean dimension of the set of generic points}

\section{Appendix: a unified approach to  Bowen entropy-like quantities and  variational principles}

Our primary motivation is to prove Theorem \ref{thm 1.4} by adapting  Feng-Huang’s approach \cite{fh12}. We further observe that this approach  has been widely used to establish the variational principles for  Bowen entropy-like quantities within  the framework of different dynamical models. See \cite{xz18, zc16, w21, zc21, ljz25} and the references therein. These work motivates  us the following question:

\emph{Can we develop a unified approach to  define Bowen entropy-like quantities on subsets and establish the corresponding  variational principles  for these entropy-like quantities?}

In this appendix, we develop a unified approach to formulate  Bowen entropy-like quantities.  The main results are the variational principles for Bowen entropy-like quantities presented in Theorems \ref{thm 5.5} and \ref{thm 5.13}.  As a result, our approach not only recovers many previous works, but also establishes variational principles for some Bowen entropy-like quantities in zero entropy and infinite entropy systems.

%



\subsection{Bowen entropy-like quantities of  compact metric space}

 Without imposing additional dynamics on phase spaces, we  first present how to develop a unified approach to define Bowen entropy-like quantities of compact metric spaces.

\subsubsection{Bowen entropy-like quantities on subsets}


\begin{df}
One says that $\gamma: (0,\infty)\rightarrow (0,1)$ is a scale function\footnote[5]{The condition (1) ensures the existence of  the  critical  value of $M_{\epsilon}^{s}(\gamma,Z,\xi)$ in $s$;  the condition (2) is used to prove the equivalence of Bowen entropy-like quantities and their weighted versions (see Proposition \ref{prop 5.9}).}  if it  satisfies that
\begin{itemize}
\item [(1)] $\gamma$ is a non-increasing continuous function and $\lim_{x\to \infty}\gamma(x)=0$;
\item [(2)] for every $\theta \in (0,1)$, for some $p>1$  it holds that $n^{p}\gamma(n)^{\theta}<1$ for all sufficiently large positive integer $n$.
\end{itemize}
\end{df}

For instance,  $e^{-x}$ and $e^{-x^\alpha}$ for every $\alpha>0$ are  scale functions.

Let $(X,d)$ be a compact  metric space, and  let $\xi:=\{d_n\}_{n=1}^{\infty}$ be  a sequence of  metrics  on $X$ compatible with $d$.
Let  $Z\subset X$ be a non-empty subset, $\epsilon >0, N\in\mathbb{N}$, $s\geq 0$ and $\gamma$ be a scale function. Put 
\begin{align}\label{ine 5.1}
M_{N,\epsilon}^{s}(\gamma,Z,\xi)=\inf\{\sum_{i\in I} \gamma(n_i)^s\},
\end{align}
where the infimum  ranges over all  finite or countable   covers $\{B(x_i, d_{n_i},\epsilon)\}_{i\in I}$ of $Z$ with $x_i\in X$ and  $n_i \geq N.$

We remark that, analogous to the definition of Hausdorff dimension,  the scale function $\gamma$ can be  regarded as  a   measurement by assigning  a geometric radius $\gamma(n_i)$  to  the open ball $B(x_i, d_{n_i},\epsilon)$.

Since $M_{N,\epsilon}^{s}(\gamma,Z,\xi)$ is non-decreasing when $N$ increases,  the limit
$$M_{\epsilon}^{s}(\gamma,Z,\xi):=\lim\limits_{N\to \infty}M_{N,\epsilon}^{s}(\gamma,Z,\xi)$$ exists. By the definition of $\gamma$, one can check that there is a critical value  of the  parameter $s$ such  that $M_{\epsilon}^{s}(\gamma,Z,\xi)$  jumps from $\infty$ to $0$, i.e.,
\begin{align*}
h_{top}^B(\gamma,Z,\xi,\epsilon)&:=\inf\{s>0:M_{\epsilon}^{s}(\gamma,Z,\xi)=0\}\\
&=\sup\{s>0:M_{\epsilon}^{s}(\gamma,Z,\xi)=\infty\}.
\end{align*}

\begin{df}\label{df 5.1}
We define the Bowen  topological entropy of $Z$  w.r.t. $\gamma$ and $\xi$ as
$$h_{top}^B(\gamma,Z,\xi)=\lim_{\epsilon \to 0}h_{top}^B(\gamma,Z,\xi,\epsilon),$$
and  define the Bowen upper  and lower metric mean dimensions  of $Z$ w.r.t. $\gamma$ and $\xi$ as
\begin{align*}
\overline{\rm mdim}_M^B(\gamma,Z,\xi)&=\limsup_{\epsilon \to 0}\frac{h_{top}^B(\gamma,Z,\xi,\epsilon)}{\log \frac{1}{\epsilon}},\\
\underline{\rm mdim}_M^B(\gamma,Z,\xi)&=\liminf_{\epsilon \to 0}\frac{h_{top}^B(\gamma,Z,\xi,\epsilon)}{\log \frac{1}{\epsilon}},
\end{align*}
respectively.
\end{df}
If $Z=\emptyset$, we let $h_{top}^B(\gamma,Z,\xi)=\overline{\rm mdim}_M^B(\gamma,Z,\xi)=0$.
\begin{rem}
The  above definition depends on the  inputs $\gamma$ and $\xi$, without involving any dynamics on $X$. 
\end{rem}

\subsubsection{Lower Brin-Katok entropy-like quantities of Borel probability measures}


Now fix the above notations $\gamma$ and $\xi$.
For $\mu \in \mathcal{M}(X)$ and $\epsilon >0$, since  $x\in X\mapsto \mu (B(x,d_n,\epsilon))$ is lower semi-continuous for every $n$, the function
$$\underline{h}_{\mu}^{BK}(\gamma,x,\xi,\epsilon):=\liminf_{n\to \infty}\frac{\log\mu (B(x,d_n,\epsilon))}{\log \gamma(n)}$$
is Borel measurable.  

\begin{df}\label{df 5.4}  
We define the lower Brin-Katok $\epsilon$-entropy of $\mu$ w.r.t. $\gamma$ and $\xi$ as
\begin{align*}
	\underline{h}_{\mu}^{{BK}}(\gamma,\xi, \epsilon):=\int_X \underline{h}_{\mu}^{BK}(\gamma,x,\xi,\epsilon)d\mu(x),
\end{align*}
and  define the lower   Brin-Katok   entropy of $\mu$ as  $$\underline{h}_{\mu}^{{BK}}(\gamma,\xi)=\lim_{\epsilon\to 0}\underline{h}_{\mu}^{{BK}}(\gamma,\xi, \epsilon).$$
\end{df}


Recall that  $\mathcal{M}_K(X):=\{\mu\in \mathcal{M}(X): \mu(K)=1\}$ and  $\mathcal{M}_{K}(\mu)=\{(\mu_\epsilon)_{\epsilon >0}: \mu_\epsilon \in  \mathcal{M}_K(X), \mu_{\epsilon} \to \mu~\text{as}~\epsilon \to 0\}$.

\begin{df}\label{df 5.5} 
For $\mu \in \mathcal{M}_K(X)$,  the  mean Brin-Katok entropies of $\mu$ is defined by
\begin{align*}
\overline{\rm mdim}_{\mu}^{BK}(\gamma,\xi)&=\sup_{(\mu_\epsilon)\in \mathcal{M}_{K}(\mu)}\{\limsup_{\epsilon \to 0}\frac{\underline{h}_{\mu_\epsilon}^{{BK}}(\gamma,\xi, \epsilon)}{\logf} \}.
\end{align*}
\end{df}

\subsubsection{Variational principle for Bowen entropy-like quantities}

\begin{thm}\label{thm 5.5}
Let $\gamma: (0,\infty) \rightarrow (0,1)$  be a scale function, and let $(X,d)$ be a compact metric space. Assume that $\xi:=\{d_n\}$ is a  sequence of  metrics  on $X$ compatible with $d$. Then for every non-empty  compact set $K$ of $X$,
\begin{align*}
h_{top}^B(\gamma,K,\xi)&=\sup\{\underline{h}_{\mu}^{{BK}}(\gamma,\xi):\mu \in \mathcal{M}(X),\mu(K)=1\},\\
\overline{\rm mdim}_M^B(\gamma,K,\xi)&=\max\{\overline{\rm mdim}_{\mu}^{BK}(\gamma,\xi):\mu \in \mathcal{M}(X),\mu(K)=1\},\\
&= \limsup_{\epsilon \to 0}\frac{1}{\logf}\sup_{\mu \in \mathcal{M}_K(X)}\underline{h}_\mu^{BK}(\gamma,\xi,\epsilon).
\end{align*}
\end{thm}

By  a minor modification of the proof  of  Lemma \ref{lem 3.5}, the  $\epsilon$-Bowen topological entropy is bounded below by the lower Brin-Katok $\epsilon$-entropy.

\begin{lem}\label{lem 5.6}
Let $(X,d)$ be a  compact metric space. Let $\mu \in \mathcal{M}(X)$ and $\epsilon >0$. Assume that $0\leq  \alpha<\infty$. If  $Z$ is Borel set with $\mu(Z)>0$ such that $\underline{h}_\mu^{BK}(\gamma,x,\xi,\epsilon)\geq \alpha$ for all $x\in Z$, then  $h_{top}^B(\gamma,Z,\xi,\frac{\epsilon}{2})\geq \alpha$.  

Consequently,  for every  Borel set $Z$  with $\mu(Z)=1$, we have  $$\underline{h}_\mu^{BK}(\gamma,\xi,\epsilon)\leq h_{top}^B(\gamma,Z,\xi,\frac{\epsilon}{2}).$$
\end{lem}

To get an upper bound of Bowen topological entropy and Bowen  metric mean dimension, we shall use the Riesz representation theorem  to  obtain a Borel probability measure on $X $ with $\mu(K)=1$. This is proceeded by introducing a  weighted  version of  Bowen entropy-like  quantities whose definition are equivalent to  Definition \ref{df 5.1}. 

Let $f: X \rightarrow \mathbb{R}$ be a bounded function, and let  $N\in \mathbb{N}$, $\epsilon>0,  s \geq 0$.  Put
$$W_{N,\epsilon}^s(\gamma,f,\xi)=\inf\{\sum_{i\in I}\limits  c_i\gamma(n_i)^s\},$$
where the infimum  ranges over all  finite or countable families  $\{B(x_i,d_{n_i},\epsilon)\}_{i\in I}$ with  $0<c_i<\infty$, $x_i\in X$  and  $n_i\geq N$,  such that 
$\sum_{i\in I}\limits c_i\chi_{B(x_i,d_{n_i},\epsilon)}\geq f.$ 

For  a non-empty subset $Z$, we let $W_{N,\epsilon}^s(\gamma,Z,\xi):=W_{N,\epsilon}^s(\gamma,\chi_Z,\xi)$  and $W_{\epsilon}^s(\gamma,Z,\xi)=\lim\limits_{N\to \infty}W_{N,\epsilon}^s(\gamma,Z,\xi).$
Likewise,  by the definition of $\gamma$ one can check that there is a critical value  of  the parameter $s$  such that $W_{\epsilon}^s(\gamma,Z,\xi)$  jumps from $\infty$ to $0$, i.e.,
\begin{align*}
h_{top}^{WB}(\gamma,Z,\xi,\epsilon):&=\inf\{s>0:W_{\epsilon}^s(\gamma,Z,\xi)=0\}\\
&=\sup\{s>0:W_{\epsilon}^s(\gamma,Z,\xi)=\infty\}.
\end{align*}

\begin{df}
We define the weighted Bowen  topological entropy of $Z$  w.r.t. $\gamma$ and $\xi$ as
$$h_{top}^{WB}(\gamma,Z,\xi)=\lim_{\epsilon \to 0}h_{top}^{WB}(\gamma,Z,\xi,\epsilon),$$
and  define  the weighted Bowen upper and lower metric mean dimensions  of $Z$ w.r.t. $\gamma$ and $\xi$ as
\begin{align*}
\overline{\rm mdim}_M^{WB}(\gamma,Z,\xi)&=\limsup_{\epsilon \to 0}\frac{h_{top}^{WB}(\gamma,Z,\xi,\epsilon)}{\log \frac{1}{\epsilon}},\\
\underline{\rm mdim}_M^{WB}(\gamma,Z,\xi)&=\liminf_{\epsilon \to 0}\frac{h_{top}^{WB}(\gamma,Z,\xi,\epsilon)}{\log \frac{1}{\epsilon}},
\end{align*}
respectively.
\end{df}

 The following  $5r$-covering lemma  can be applied  to  verify  that the aforementioned  Bowen entropy-like quantities coincide with the weighted quantities. 

\begin{lem}\rm{\cite[Lemma 6.3]{w21}}\label{lem 5.7}
Let $(X,d)$ be a compact metric space, and let $r>0$. For every  family $\mathcal{F}=\{B(x_i,d,r)\}_{i\in I}$   of  $d$-open balls of $X$, there exists a finite  subfamily $\mathcal{G}:=\{B(x_j,d,r)\}_{j\in J}$ such  that
$$\cup_{i\in I}~B(x_i,d,r)\subset \cup_{j\in J}~B(x_j,d,5r),$$
and  $I_j\cap I_k=\emptyset$  for any distinct $j,k\in J$, where  $I_j:=\{i\in I:B(x_i,d,r)\cap B(x_j,d,r)\not= \emptyset\}.$
\end{lem}

\begin{prop}\label{prop 5.9}
Let $Z\subset X$ be a non-empty subset.  Then for any $s\geq 0, \epsilon >0, \delta  \in (0,1)$ and sufficiently large  $N$, 
$$M_{N,5\epsilon}^{s+\delta}(\gamma,Z,\xi)\leq W_{N,\epsilon}^s(\gamma,Z,\xi)\leq M_{N,\epsilon}^s(\gamma,Z,\xi).$$
Consequently, we have $h_{top}^{B}(\gamma,Z,\xi)=h_{top}^{WB}(\gamma,Z,\xi)$ and 
\begin{align*}
\overline{\rm{mdim}}_M^{B}(\gamma,Z,\xi)=\overline{\rm mdim}_M^{WB}(\gamma,Z,\xi),~
\underline{\rm{mdim}}_M^{B}(\gamma,Z,\xi)=\underline{\rm mdim}_M^{WB}(\gamma,Z,\xi).
\end{align*} 
\end{prop}
\begin{proof}
It is clear that $$W_{N,\epsilon}^s(\gamma,Z,\xi)\leq M_{N,\epsilon}^s(\gamma,Z,\xi)$$   by letting $c_i=1$  for every finite or countable cover  $\{B(x_i,d_{n_i},\epsilon)\}_{i\in I}$ of $Z$.

  We adapt the arguments in \cite[Proposition 6.4]{w21} and \cite[Proposition 3.3]{ycz24} to show $$M_{N,5\epsilon}^{s+\delta}(\gamma,Z,\xi)\leq W_{N,\epsilon}^s(\gamma,Z,\xi).$$ Fix $\epsilon>0$ and $\delta \in (0,1)$.  By the definition of $\gamma$,  there exists  $p>1$ such that $\gamma(N)^{\delta}<\frac{1}{N^{p}}$ for all sufficiently large $N$. Moreover, we may require that $\sum_{n\geq N}\frac{1}{n^p}<1$ for such $N$.  Let  $\{(B(x_i,d_{n_i}, \epsilon),c_i)\}_{i \in I}$   be a  finite or countable family with  $0<c_i<\infty$, $x_i \in X$, $n_i \geq N$  satisfying
$\sum_{i\in I}\limits c_i \chi_{B(x_i,d_{n_i}, \epsilon)}\geq \chi_Z$. For every $n\geq N$ and $0<t<1$, we define the following sets:
$$I_n=\{i\in I:n_i=n\}$$ 
and 
\begin{align*}
Z_{n,t}&=\{z\in Z:  \sum_{i\in I_n}\limits c_i \chi_{B(x_i,d_{n}, \epsilon)}(z)> t \},\\
I_n^t&=\{i\in I_n: B(x_i,d_{n}, \epsilon)\cap Z_{n,t}\not=\emptyset\}.
\end{align*}
$$.$$
Then   $Z_{n,t}\subset \cup_{i\in I_n^t} B(x_i,d_{n}, \epsilon).$ We apply Lemma \ref{lem 5.7}  to the family $\mathcal{F}=\{B(x_i,d_{n}, \epsilon)\}_{i \in I_n^t}$. There exists a finite subfamily of $\mathcal{F}$, indexed by   $J$, such  that
$$\cup_{i\in {I_{n}^t}}  B(x_i,d_{n},\epsilon) \subset  \cup_{j\in J} B(x_j,d_{n},5\epsilon)$$
and  $(I_n^t)_j\cap (I_n^t)_k=\emptyset$  for any distinct $j,k\in J$, where  $(I_n^t)_j:=\{i\in I_n^t:B(x_i,d_n,\epsilon)\cap B(x_j,d_n,\epsilon)\not= \emptyset\}.$  Choose $y_j \in B(x_j,d_n,\epsilon)\cap Z_{n,t} $ for each $j\in J$.  Then,  it follows from the  choice of $(I_n^t)_j$ that $$\sum_{i\in I_n^t}\limits c_i \chi_{B(x_i,d_n,\epsilon)}(y_j)= \sum_{i\in (I_n^t)_j}\limits c_i >  t.$$ This implies that   
$$\#J< \frac{1}{t}\sum_{j\in J} \sum_{i\in (I_n^t)_j}\limits c_i \leq  \frac{1}{t}\sum_{i\in I_n^t}\limits c_i $$
by summing  over every  $j \in J$, and hence $$M_{N,5\epsilon}^{s+\delta} (\gamma,Z_{n,t},\xi) \leq \#J\cdot \gamma(n)^{(s+\delta)}\leq  \frac{1}{n^p t} \sum_{i\in I_n}\limits c_i \gamma(n)^{s}.$$  By the relation $Z=\cup_{n\geq N}Z_{n,\frac{1}{n^p}t}$, we have  $$M_{N,5\epsilon}^{s+\delta}(\gamma,Z,\xi)\leq \sum_{n\geq N}M_{N,5\epsilon}^{s+\delta} (\gamma,Z_{n,\frac{1}{n^p}t},\xi)\leq \frac{1}{t} \sum_{i\in I}\limits c_i \gamma(n_i)^{s}.$$
Taking infimum over  such families  $\{(B(x_i,d_{n_i}, \epsilon),c_i)\}_{i \in I}$, and then letting $t\to 1$,  this completes the proof.
\end{proof}

Modifying  the techniques  from  \cite[Lemma 3.4]{fh12}, we can  obtain a   dynamical version of Frostman's lemma. 

\begin{lem}\label{lem 5.9}
Let $K  \subset X$ be a non-empty compact subset, and let $ s >0, \epsilon >0,$ and $ N\in \mathbb{N}$. If $c:=W_{N,\epsilon}^s(\gamma,K,\xi)>0$,  then there exists  $\mu \in \mathcal{M}(X)$ with  $\mu (K)=1$ satisfying 
$$\mu(B(x,d_n,\epsilon))\leq \frac{1}{c}\cdot \gamma(n)^s$$
for all $x\in X$ and $ n\geq N$. 
\end{lem}

\begin{proof}[Proof of Theorem \ref{thm 5.5}]
We divide the proof into two steps:

\emph{Step 1:} We first show 
\begin{align*}
h_{top}^{B}(\gamma,K,\xi)=\sup\{\underline{h}_\mu^{BK}(\gamma,\xi): \mu \in \mathcal{M}(X),\mu(K)=1\}.
\end{align*}

By  Lemma \ref{lem 5.6}, it is clear that for every $\epsilon >0$,
\begin{align}\label{inequ 5.1}
\sup\{\underline{h}_\mu^{BK}(\gamma,\xi,\epsilon): \mu \in \mathcal{M}(X),\mu(K)=1\}\leq h_{top}^B(\gamma,K,\xi,\frac{\epsilon}{2}).
\end{align}

On  the other hand, let  $0<s< h_{top}^{WB}(\gamma,K,\xi,\epsilon)$. Then $c:= W_{N_0,\epsilon}^s(\gamma,K,\xi)>0$ for some $N_0$. Using Lemma \ref{lem 5.9},  there exists  $\mu \in \mathcal{M}(X)$ with  $\mu (K)=1$ satisfying 
$$\mu(B(x,d_n,\epsilon))\leq \frac{1}{c}\cdot \gamma(n)^s$$
for all $x\in X$ and $ n\geq N_0$. 
This shows that $\underline{h}_\mu^{BK}(\gamma,\xi,\epsilon)\geq s$ and 
\begin{align*}
 h_{top}^{WB}(\gamma,K,\xi,\epsilon) \leq \sup\{\underline{h}_\mu^{BK}(\gamma,\xi,\epsilon): \mu \in \mathcal{M}(X),\mu(K)=1\}.
\end{align*}
by letting $s\to h_{top}^{WB}(\gamma,K,\xi,\epsilon)$. Hence, by    Proposition \ref{prop 5.9},  we get
\begin{align}\label{inequ 5.3}
h_{top}^{B}(\gamma,K,\xi,5\epsilon) \leq \sup\{\underline{h}_\mu^{BK}(\gamma,\xi,\epsilon): \mu \in \mathcal{M}(X),\mu(K)=1\}.
\end{align}
 
By (\ref{inequ 5.1}) and (\ref{inequ 5.3}),  noting that $\lim_{\epsilon \to 0}h_{top}^{B}(\gamma,K,\xi,\epsilon)=\sup_{\epsilon >0}h_{top}^{B}(\gamma,K,\xi,\epsilon),$ we have 
\begin{align*}
h_{top}^{B}(\gamma,K,\xi)&=\lim_{\epsilon \to 0}~ \sup\{\underline{h}_\mu^{BK}(\gamma,\xi,\epsilon): \mu \in \mathcal{M}(X),\mu(K)=1\},\\
&=\sup\{\underline{h}_\mu^{BK}(\gamma,\xi): \mu \in \mathcal{M}(X),\mu(K)=1\}.
\end{align*}

Moreover,  passing the upper (or lower) limit to the inequalities (\ref{inequ 5.1}) and (\ref{inequ 5.3}), we also  establish the variational principles for Bowen metric mean dimension:
\begin{align*}
	\overline{\rm mdim}_M^B(\gamma,K,\xi)&= \limsup_{\epsilon \to 0}\frac{1}{\logf}\sup_{\mu \in \mathcal{M}_K(X)}\underline{h}_\mu^{BK}(\gamma,\xi,\epsilon),\\
	\underline{\rm mdim}_M^B(\gamma,K,\xi)&=\liminf_{\epsilon \to 0}\frac{1}{\logf}\sup_{\mu \in \mathcal{M}_K(X)}\underline{h}_\mu^{BK}(\gamma,\xi,\epsilon).
\end{align*}

\emph{Step 2:}  Next  we prove that
$$\overline{\rm mdim}_M^B(\gamma,K,\xi)=\max\{\overline{\rm mdim}_{\mu}^{BK}(\gamma,\xi):\mu \in \mathcal{M}(X),\mu(K)=1\}.$$

Let $\mu \in \mathcal{M}_K(X)$. If $ \mu_{\epsilon} \to \mu$ with $\mu_\epsilon \in  \mathcal{M}_K(X)$ as $\epsilon \to 0$, then, by (\ref{inequ 5.1}) we have  $$\limsup_{\epsilon \to 0}\frac{\underline{h}_{\mu_\epsilon}^{{BK}}(\gamma,\xi, \epsilon)}{\logf}\leq  \overline{\rm mdim}_M^B(\gamma,K,\xi).$$ Taking  the supremum  over  $(\mu_{\epsilon}) \in \mathcal{M}_K(\mu)$  and then  over  $\mathcal{M}_K(X)$, we have
$$ \sup\{\overline{\rm mdim}_{\mu}^{BK}(\gamma,\xi):\mu \in \mathcal{M}(X),\mu(K)=1\}\leq \overline{\rm mdim}_M^B(\gamma,K,\xi).$$

By (\ref{inequ 5.3}), for every $\epsilon>0$ there exists $\mu_{\epsilon} \in \mathcal{M}_K(X)$ such that 
\begin{align}\label{inequ 4.5}
h_{top}^{B}(\gamma,K,\xi,5\epsilon)-1 <\underline{h}_{\mu_{\epsilon}}^{BK}(\gamma,\xi,\epsilon).
\end{align}
Repeating the similar arguments  as the proof of  the inequality (\ref{inequ 2.13}), there  exists  $\mu_0 \in \mathcal{M}_K(X)$ such that 
 \begin{align*}
\overline{\rm mdim}_M^B(\gamma,K,\xi)  \leq \overline{\rm mdim}_{\mu_0}^{BK}(\gamma,\xi),
 \end{align*}
where $\mu_0$  attains the supremum of the variational principle. 
\end{proof}

\subsection{Bowen entropy-like quantities of $G$-systems}

We  continue to present how to develop a unified approach to  formulate Bowen entropy-like quantities of $G$-systems.

 A \emph{$G$-system} is a pair $(X,G)$, where    $G$ is an infinite  topological semi-group (or group) with an identity element $e_G$,  $(X,d)$ is a compact metric space, and $\Gamma: G\times X \rightarrow X$ is a continuous map satisfying: 
\begin{enumerate}
	\item$\Gamma(e_G,x)=x$ for every $x\in X$; 
	\item  $\Gamma(g_1,\Gamma(g_2,x))=\Gamma(g_1g_2,x)$ for every $g_1,g_2\in G$ and $x\in X$.
\end{enumerate}

We write $gx:=\Gamma(g,x)$, and let $|F|$ be the cardinality of $F$.
  For every non-empty finite subset $ F$ of $G$ containing $e_{G}$, we define a Bowen metric\footnote[6]{If $G$ is a group,  $T_g: X\rightarrow X$ is a homeomorphism map for all $g\in G$; in this case,  the Bowen metric $d_F$ is defined for any non-empty finite subset $F$ of $G$ without  requiring $e_G \in F$.} on $X$ $$d_F(x,y)=\max_{g\in F}d(gx,gy).$$ 

Let  $\{F_n\}$ be a sequence of  non-empty finite subsets of $G$ containing $e_G$  with $\lim_{n \to \infty} |F_n|=\infty$.  One says that a scale function $\gamma: (0,\infty)\rightarrow (0,1) $ is \emph{compatible} with $\{F_n\}$  
if it  satisfies that
\begin{itemize}
	\item [(1)] $\gamma$ is a non-increasing continuous function and $\lim_{x\to \infty}\gamma(x)=0$;
	\item [(2)] for every $\theta \in (0,1)$, for some $p>1$  it holds that $|F_n|^{p}\gamma{(|F_n|)}^{\theta}<1$ for all sufficiently large positive integer $n$.
\end{itemize}

For instance, letting $G=\mathbb{Z}^d$  and  $\{F_n\}=\{0,1,...,n-1\}^d$ be   a subset of $G$, then $\gamma(x)=e^{-x}$ is a scale  function compatible with   $\{F_n\}$.
 
From now on, fix  a scale function $\gamma$ compatible with the sequence $\{F_n\}$ and  a sequence of  metrics $\xi:=\{d_{F_n}\}$   on $X$ compatible with $d$. 
Similar to  Definitions \ref{df 5.1}, \ref{df 5.4} and  \ref{df 5.5}, we  define the following quantities\footnote[7]{Letting $d_{n_i}:=d_{F_{n_i}}$, readers may  recognize that we do not necessarily  define these new Bowen entropy-like quantities for $G$-systems. Actually, beyond $\mathbb{Z}$-actions,  it is  ``popular" to  define the geometric  radius  of a Bowen ball  using the cardinality  of $F_{n_i}$ rather than $n_i$ itself. This results in a scenario that is not always covered by the general framework presented in subsection 4.1, so special treatment of this case is justified.} by replacing  $B(x_i, d_{{n_i}},\epsilon)$, $\gamma(n_i)$ with $B(x_i, d_{F_{n_i}},\epsilon)$,    $\gamma|_{\{F_n\}}:=\gamma(|F_{n_i}|)$, respectively.

\begin{df}
The Bowen  topological entropy of $Z$  w.r.t. $\gamma|_{\{F_n\}}$ and $\xi$ is defined by 
$$h_{top}^B(\gamma|_{\{F_n\}},Z,\xi)=\lim_{\epsilon \to 0}h_{top}^B(\gamma|_{\{F_n\}},Z,\xi,\epsilon),$$
and the Bowen upper and lower metric mean dimensions  of $Z$ w.r.t. $\gamma|_{\{F_n\}}$ and $\xi$ are respectively defined by
\begin{align*}
\overline{\rm mdim}_M^B(\gamma|_{\{F_n\}},Z,\xi)&=\limsup_{\epsilon \to 0}\frac{h_{top}^B(\gamma|_{\{F_n\}},Z,\xi,\epsilon)}{\log \frac{1}{\epsilon}},\\
\underline{\rm mdim}_M^B(\gamma|_{\{F_n\}},Z,\xi)&=\liminf_{\epsilon \to 0}\frac{h_{top}^B(\gamma|_{\{F_n\}},Z,\xi,\epsilon)}{\log \frac{1}{\epsilon}}.
\end{align*}
\end{df}

\begin{df}
For  $\mu \in \mathcal{M}(X)$, we define the lower   Brin-Katok   entropy of $\mu$ w.r.t. $\gamma|_{\{F_n\}}$ and $\xi$ as  $$\underline{h}_{\mu}^{{BK}}(\gamma|_{\{F_n\}},\xi)=\lim_{\epsilon\to 0}\underline{h}_{\mu}^{{BK}}(\gamma|_{\{F_n\}},\xi, \epsilon).$$
For $\mu \in \mathcal{M}_K(X)$,  the  mean Brin-Katok entropies of $\mu$ w.r.t. $\gamma|_{\{F_n\}}$ and $\xi$ is defined by
\begin{align*}
	\overline{\rm mdim}_{\mu}^{BK}(\gamma|_{\{F_n\}},\xi)&=\sup_{(\mu_\epsilon)\in \mathcal{M}_K(\mu)}\{\limsup_{\epsilon \to 0}\frac{\underline{h}_{\mu_\epsilon}^{{BK}}(\gamma|_{\{F_n\}},\xi, \epsilon)}{\logf} \}.
\end{align*}
\end{df}

The proof of Theorem \ref{thm 5.5} also works for  the following:

\begin{thm}\label{thm 5.13}
 Let $(X,G)$ be a $G$-system. Suppose that $\gamma: (0,\infty)\rightarrow (0,1)$  is a scale function compatible with the sequence $\{F_n\}$, and  $\xi:=\{d_{F_n}\}$ is a  sequence of metrics on $X$  compatible with $d$. Then for every non-empty  compact set $K$ of $X$,
\begin{align*}
h_{top}^B(\gamma|_{\{F_n\}},K,\xi)&=\sup\{\underline{h}_{\mu}^{{BK}}(\gamma|_{\{F_n\}},\xi):\mu \in \mathcal{M}(X),\mu(K)=1\},\\
\overline{\rm mdim}_M^B(\gamma|_{\{F_n\}},K,\xi)&=\max\{\overline{\rm mdim}_{\mu}^{BK}(\gamma|_{\{F_n\}},\xi):\mu \in \mathcal{M}(X),\mu(K)=1\},\\
&= \limsup_{\epsilon \to 0}\frac{1}{\logf}\sup_{\mu \in \mathcal{M}_K(X)}\underline{h}_\mu^{BK}(\gamma|_{\{F_n\}},\xi,\epsilon).
\end{align*}
\end{thm}

\subsection{The application of main results}
In this subsection, by  clarifying the scale function $\gamma$ and the  sequence $\xi$ of compatible metrics on the phase space, we demonstrate how to apply the main theorems to recover the existing results for Bowen entropy-like quantities, and derive new variational principles for other types of Bowen entropy-like quantities within the framework of dynamical systems.

$(1)$. Let $(X,d,T)$ be a TDS.  Recall that $d_n(x,y)=\max_{0\leq j \leq n-1}d(T^jx,T^jy)$ is the Bowen metric on $X$, and  $d_n^{\alpha}(x,y)= \max_{0\leq j \leq n-1}e^{\alpha j}d(T^jx,T^jy)$, where $\alpha \geq 0$, is the $\alpha$-Bowen metric on $X$.  Theorem \ref{thm 5.5} recovers the previous   variational principles for Bowen topological entropy and Bowen metric mean dimension: 
\begin{table}[htbp]
	\centering
	\begin{tabular}{lcl}  
		\toprule  
		Variational principle & Scale function & Metric   \\
			\toprule
		Bowen topological entropy \cite[Theorem 1.2]{fh12} & $e^{-x}$ & $d_n(x,y)$    \\
		$\alpha$ estimation entropy \cite[Theorem  3.1]{zc21b}  & $e^{-x}$ & $ d_n^{\alpha}(x,y)$   \\
		Bowen metric mean dimension \cite[Theorem 1.4]{w21} & $e^{-x}$ & $ d_n(x,y)$   \\
		FK Bowen metric mean dimension \cite[Theorem 1.1]{GZ24}  & $e^{-x}$ & $ d_{FK_n}(x,y)$   \\
		\bottomrule  
	\end{tabular}
\end{table}

Besides, we have a new variational principle for Bowen metric mean dimension whose form is close to the classical ones. More precisely, let $\gamma(x)=e^{-x}$  be the scale function and  $\xi:=\{d_{n}\}$ with $d_n(x,y)=\max_{0\leq j \leq n-1}d(T^jx,T^jy)$ as  in Theorem \ref{thm 5.5}, and let $\overline{\rm mdim}_M^B(T,K,d):=\overline{\rm mdim}_M^B(\gamma,K,\xi)$, $\overline{\rm mdim}_{\mu}^{BK}(T,d):=\overline{\rm mdim}_{\mu}^{BK}(\gamma,\xi)$.
\begin{thm}
Let  $(X,d,T)$ be a TDS.   Then
\begin{align*}
\overline{\rm mdim}_M^B(T,K,d)&=\max\{\overline{\rm mdim}_{\mu}^{BK}(T,d):\mu \in \mathcal{M}(X),\mu(K)=1\}.
\end{align*}
\end{thm}

Suppose that $S=\{i_n\} \subset \mathbb{Z}_{\geq 0}$ is an increasing sequence that converges to $\infty$ as $n\to \infty$.  Let  $\gamma(x)=e^{-x}$ be the scale function and  $\xi:=\{d_{n}^S\}$ with $$d_n^S(x,y)=\max_{0\leq j \leq n-1}(T^{i_j}x,T^{i_j}y)$$  as  in Theorem \ref{thm 5.5}. For systems with zero Bowen topology entropy,  analogous to the topological sequence entropy \cite{g0074}, we define $h_{top}^{S}(T,K):=h_{top}^B(\gamma,K,\xi)$  as the Bowen topological sequence entropy of $K$,  and  define $\underline{h}_{\mu}^{{S}}(T):=\underline{h}_{\mu}^{{BK}}(\gamma,\xi)$ as the lower Brin-Katok  sequence entropy of $\mu$.  Our result allows us to establish the following  variational principle for Bowen topological sequence entropy:

\begin{thm}
Let  $(X,d,T)$ be a TDS.  Then
\begin{align*}
h_{top}^{S}(T,K)&=\sup\{\underline{h}_{\mu}^{{S}}(T):\mu \in \mathcal{M}(X),\mu(K)=1\}.
\end{align*}
\end{thm}

$(2)$. Let  $(X,d,\phi)$ be a flow. Considering the scale function $\gamma(x)=e^{-x}$ and $d_n(x,y)=\max_{s\in [0,n]}\limits d(\phi_sx,\phi_sy)$,  Theorem \ref{thm 1.4}  directly follows from Theorem \ref{thm 5.5}.

$(3)$.  Let a pair $(X,\phi)$ denote a non-autonomous dynamical system, where $(X,d)$ is  a  compact metric space and $\phi: [0,\infty)\times X \rightarrow X$ is a continuous map with $\phi(0,x)=x$ for all $x\in X$. For every $n\in \mathbb{N}$, we  define
$$d_n^{\phi}(x,y)=\max_{s\in [0,n]}d(\phi(s,x), \phi(s,y)),$$
and $d_n^{\Phi}(x,y)=\max_{0\leq j<n}d(\phi(j,x),\phi(j,y))$. Letting $\gamma(x)=e^{-x}$,  by Theorem  $\ref{thm 5.5}$ we  recover the  variational principles for  Bowen  topological entropy \cite[Theorems 1.2 and  1.4]{xz18}.

$(4)$. Let $X$ be a compact metric space with a family $\mathcal{G}=\{\text{id}_X,f_1,...,f_k\}$ of continuous self-maps on $X$. A free  semigroup $G$  continuously acting  on  $X$  is finitely generated with the generator $\mathcal{G}$ if $G=\cup_{n\in  \mathbb{N}} G_n$, where
$$G_n=\{g_1\circ  \cdots \circ g_n: g_j \in \mathcal{G}\}.$$ 
For every $n\in \mathbb{{N}}$, we define $d_n(x,y):=\max_{g\in  G_n}d(gx,gy)$.  Letting  $\gamma(x)=e^{-x}$, we  recover the variational principle for Bowen topological entropy \cite[Theorem 1.2]{xm22}.

In \cite{zc21}, the  terminology for  ``finitely generated" is  a   slightly different from the one given in 
\cite{xm22}. It means that $G=\cup_{n\geq 1}\widetilde{G}_n$, where $$ \widetilde{G}_n:=\{g_1\circ \cdots \circ g_m: 1\leq m<n,g_j \in  \mathcal{G}\backslash\{\text{id}_X\}\}\cup\{\text{id}_X\}.$$
Denote by $|\widetilde{G}_n|:=\sum_{j=0}^{n-1}k^j$  the cardinality\footnote[8]{It is possible that two distinct compositions of self-maps chosen from the generated set yield the same function. For the purpose of calculating the topological entropy of semi-group action, it is common to consider all possible compositions, rather than counting the cardinality of distinct results of these compositions.} of  $\widetilde{G}_n$. Let $\gamma(|\widetilde{G}_n|)=e^{-|\widetilde{G}_n|}$ and $d_{\widetilde{G}_n}(x,y)=\max_{g\in \widetilde{G}_n}d(gx, gy)$. Then Theorem \ref{thm 5.13}  recovers the variational  principle for Bowen topological entropy  \cite[Theorem 3.1]{zc21}.

$(5)$.  By Theorem  \ref{thm 5.13}, we have the following variational principle for Bowen topology entropy of  $G$-system, which  extends the variational principle for Bowen topological entropy of  amenable groups  \cite[Theorem 3.1]{zc16} to  general group actions.

\begin{thm}\label{thm 5.16}
Let $(X,d)$ be a compact metric space and $G$  be an infinite topological group. Let $(X,G)$ be $G$-system, and let  $\{F_n\}$ be a  sequence  of non-empty finite subsets of $G$ satisfying $\lim_{n \to \infty}$ $\frac{|F_n|}{\log n}=\infty$.  Assume that $\gamma(x)=e^{-x}$  is  a scale function, and  $\xi:=\{d_{F_n}\}$ is a  sequence of metrics on $X$   compatible with $d$. Then for every non-empty  compact set $K$ of $X$,
\begin{align*}
h_{top}^B(G,K)=\sup\{\underline{h}_{\mu}^{{BK}}(G):\mu \in \mathcal{M}(X),\mu(K)=1\},
\end{align*}
where $h_{top}^B(G,K):=h_{top}^B(\gamma|_{\{F_n\}},K,\xi)$ denotes the Bowen topological topology of $K$, and  $\underline{h}_{\mu}^{{BK}}(G):=\underline{h}_{\mu}^{{BK}}(\gamma|_{\{F_n\}},\xi)$ denotes the lower Brin-Katok entropy of $\mu$.
\end{thm}





\section*{Acknowledgement} 

 The first author  was  supported by  the China Postdoctoral Science Foundation (No. 2024M763856) and  the Postdoctoral Fellowship Program of CPSF  (No. GZC20252040).he second author was supported by the National
 Natural Science Foundation of China (No.12471184 ). The third author
 was supported by the National Natural Science Foundation of China (No. 11971236) and
 Qinglan project of Jiangsu Province of China.



\end{document}